\documentclass[11pt,leqno]{amsart}

\usepackage{graphicx}
\usepackage[margin=1.3in]{geometry}
\usepackage{amsmath}
\usepackage{amssymb}
\usepackage{amsthm}
\usepackage{amsfonts}
\usepackage{mathrsfs}
\usepackage[mathscr]{euscript}
\usepackage[latin1]{inputenc}
\usepackage[numbers]{natbib}
\usepackage[usenames, dvipsnames]{color}
\usepackage{enumerate}
\usepackage[scriptsize,hang,raggedright]{subfigure}
\usepackage[cmyk]{xcolor}
\usepackage{mathtools}
\usepackage{graphicx,color,amsmath,amssymb,mathrsfs}
\usepackage{amsfonts,amsthm,epsfig,epstopdf}

\usepackage{hyperref}
\usepackage{pdfsync}
\usepackage{dsfont}

\usepackage{bbm}
\usepackage{multirow}

\allowdisplaybreaks

\definecolor{refkey}{gray}{.45}
\definecolor{labelkey}{gray}{.45}

\definecolor{grey}{rgb}{.7,.7,.7}

\newtheorem{theorem}{Theorem}[section]
\newtheorem{proposition}[theorem]{Proposition}
\newtheorem{lemma}[theorem]{Lemma}

\theoremstyle{remark}
\newtheorem{remark}[theorem]{Remark}
\theoremstyle{definition}

\newtheorem{example}[theorem]{Example}

\def\H{\mathcal H}

\def\R{\mathbb R}

\def\e{\varepsilon}

\def\S{\Sigma}
\def\vphi{\varphi}
\def\Div{{\rm div}\,}
\def\om{\omega}
\def\l{\lambda}
\def\g{\gamma}

\def\pa{\partial}

\def\E{\mathcal{E}}

\def\F{\mathcal{P}_f}

\def\Hi{\mathcal{H}}

\renewcommand{\>}{\rangle}

\renewcommand{\Div}{{\rm div \,}}

\newcommand{\cc}{\subset\!\subset}

\newcommand{\bary}{{\rm bar}}
\newcommand{\Rn}{{\mathbb{R}^n}}

\def\H{\mathcal H}

\def\l{{\lambda}}
\newcommand{\vertiii}[1]{{\left\vert\kern-0.25ex\left\vert\kern-0.25ex\left\vert #1
    \right\vert\kern-0.25ex\right\vert\kern-0.25ex\right\vert}}
\def\S{\mathbb{S}}

\newcommand{\DIVV}{{\rm div}}

\newcommand{\na}{{\nabla}}
\def\>{{\rangle}}

\newcommand{\ba}{\begin{array}}
\newcommand{\ea}{\end{array}}

\newcommand{\tld}[1]{\widetilde{#1}}
\newcommand{\bthm}{\begin{theorem}}
\newcommand{\ethm}{\end{theorem}}
\newcommand{\bprop}{\begin{proposition}}
\newcommand{\eprop}{\end{proposition}}
\newcommand{\blemma}{\begin{lemma}}
\newcommand{\elemma}{\end{lemma}}
\newcommand{\bexmpl}{\begin{example}}
\newcommand{\eexmpl}{\end{example}}

\newcommand{\beqn}{\begin{equation}}
\newcommand{\eeqn}{\end{equation}}
\newcommand{\beqns}{\begin{equation*}}
\newcommand{\eeqns}{\end{equation*}}

\newcommand{\pt}{\partial}

\newcommand{\Hn}{\mathcal{H}^{n-1}}

\newcommand{\V}{\mathcal{V}}

\renewcommand{\leq}{\leqslant}
\renewcommand{\geq}{\geqslant}

\definecolor{mygreen}{rgb}{0.1,0.75,0.2}

\newcommand{\eps}{\epsilon}

\newcommand{\varE}{\mathcal{E}}

\newcommand{\Case}{\it{Case }}

\newcommand{\Per}{\mathcal{P}}

\newcounter{myenumi}
\newenvironment{myenumerate}{
\begin{enumerate} 
\setcounter{enumi}{\value{myenumi}}
}{
\setcounter{myenumi}{\value{enumi}}
\end{enumerate}
}

\DeclareMathOperator{\id}{Id}

\DeclareMathOperator{\diam}{diam}

\numberwithin{equation}{section}

\title[Anisotropic liquid drop models]{Anisotropic liquid drop models}

\author{Rustum Choksi}
\address{Department of Mathematics and Statistics,
				McGill University, Montr\'{e}al, QC}
\email{rustum.choksi@mcgill.ca}
\author{Robin Neumayer}
\address{School of Mathematics,
				Institute for Advanced Study, Princeton, NJ}
\email{neumayer@ias.edu}
\author{Ihsan Topaloglu}
\address{Department of Mathematics and Applied Mathematics,
				Virginia Commonwealth University, Richmond, VA}
\email{iatopaloglu@vcu.edu}
\date{\today}                                        

\begin{document}

\begin{abstract} We introduce and study certain  variants of Gamow's liquid drop model  
in which an anisotropic surface energy replaces the perimeter. After existence and nonexistence results are established, the shape of minimizers is analyzed. Under suitable regularity and ellipticity assumptions on the surface tension, Wulff shapes are minimizers in this problem if and only if the surface energy is isotropic. In sharp contrast,  Wulff shapes are the unique minimizers for certain crystalline surface tensions. We also introduce and study several  related liquid drop models with anisotropic repulsion for which the 
Wulff shape is the minimizer in the small mass regime. 
\end{abstract}

\maketitle

\baselineskip=13pt

\section{Introduction}\label{sec:intro}

Gamow's liquid drop (LD) model, early versions of which date back to 1930 (\cite{Ga1930}),  has recently generated considerable interest in the calculus of variations community (see \cite{ChMuTo2017} for a general introduction). It was initially developed to predict the mass defect curve and the shape of atomic nuclei.   In its  modern rendition,  it includes two competing forces: 
an  attractive surface energy associated with a depletion of nucleon
density near the nucleus boundary,  and  repulsive Coulombic interactions due to the presence of
positively charged protons. 
Mathematically,  it has a very simple form: over all sets $E\subset \R^3$ of measure $m$, minimize 
\[ \Per (E) \, + \, \int_{E}\!\int_E \frac{dx dy}{|x-y|}, \]
 where $\Per (E)$ denotes the perimeter in the geometric measure-theoretic sense.
 As such, the LD model is a paradigm for shape optimization via competitions of short and long-range interactions, and indeed it (or variants of it)  has been used to   model many different systems at all length scales,  from atomic (its original conception) to cosmological. 

Often studied is the generalization of the LD model in which one works in $n$ space dimensions with 
any Riesz potential; that is, for fixed $\alpha \in (0,n)$  we consider the variational problem 
\begin{equation}\label{eqn: iso drop} \quad 
\inf \, \Big\{\mathcal{E}(E):=  \Per (E)+ \V(E) \, \big\vert  \, |E| =m \Big\} \quad {\rm where} \quad 
\V(E) \, := \, \int_{E}\!\int_E \frac{dx dy}{|x-y|^\alpha}.
\end{equation}
 From a mathematical point of view, the LD model has two notable  features:  
 
 \vspace{0.1cm}
\noindent {\bf Feature (i).} Balls are extremal for each individual term but at opposite ends of the spectrum --  
 balls are best for (minimizers of) the first term but worst for (maximizers of) the second  term. In particular,  a ball of mass $m$ is always a critical point of $\mathcal{E}$ among volume-preserving variations. \\
{\bf Feature (ii).} The two terms scale differently in mass $m$, with  perimeter dominating  for small mass and repulsion dominating for large mass. 
 
 \vspace{0.1cm}
In \cite{ChPe2011}, it is conjectured that up to a critical mass $m_c$, balls are the unique minimizers,  while above $m_c$,  minimizers fail to exists. 
Due primarily to the work of Kn\"{u}pfer and Muratov (\cite{KM13,KM14}) and Figalli et al. (\cite{FFMMM}), with additional/related  contributions  from \cite{BoCr14, ChPe2010, FrLi2015,Julin14,LO14}, the state of the art for (global) minimizers of \eqref{eqn: iso drop} is as follows:  For any $n \geq 2$,  we have: 
\begin{enumerate}[(G1)] 
	\item for all $\alpha \in (0,n)$ there exists $\tld{m}_1>0$ such that if $m\leq \tld{m}_1$, then the problem admits a minimizer; \label{item:G1}
	\vspace{0.1cm}
	\item for all $\alpha \in (0,n)$ there exists $\tld{m}_0>0$, $\tld{m}_0\leq \tld{m}_1$, such that if $m \leq \tld{m}_0$, then the minimizer is uniquely (modulo translations) given by the ball of mass $m$; and  \label{item:G2}
		\vspace{0.1cm}
	\item for all $\alpha \in (0,2)$ there exists $\tld{m}_2>0$ such that if $m>\tld{m}_2$, then no minimizer exists. \label{item:G3}
	\end{enumerate} 
	\vspace{-0.4cm}
It is conjectured in \cite{ChPe2011} that $\tld{m}_0 = \tld{m}_1 =\tld{m}_2$ when $n=3$ and $\alpha= 1$. While the conjecture remains open, it was shown in \cite{BoCr14} that $\tld{m}_0 = \tld{m}_1 =\tld{m}_2$ in any dimension for $\alpha$ sufficiently small. It also remains open whether the nonexistence result (G\ref{item:G3}) can be extended to $\alpha\in [2,n)$. 
  
  \smallskip
 
In this article we introduce and discuss anisotropic variants of (\ref{eqn: iso drop}). In particular, we address 
  two classes of  anisotropic liquid drop models consisting of
\begin{itemize}
\item[(1)] anisotropic perimeter with isotropic long-range repulsions; 
\vspace{0.1cm}
\item[(2)] anisotropic perimeter with  {\it related} anisotropic long-range repulsions. 
\end{itemize}
To our knowledge this is the first mathematical treatment of these problems, particularly surprising since they are both physically and mathematically well-motivated. On the physical side, it is natural to consider 
surface energies which are anisotropic (cf. \cite{Mackie75}). Indeed, at the microscopic level, the existence of a tensor force can produce an asymmetry in the nucleon-nucleon potential, creating an anisotropic surface tension. 
From the more macroscopic perspective, surface diffuseness can vary across the interface boundary, also 
creating an anisotropic surface tension. In such situations, it is natural to couple the anisotropic surface energy with an 
isotropic (e.g. Coulombic) long-range interactions due to the presence of charged particles. Thus we arrive at class (1). On the other hand, atomic lattice structures as seen, for example, in 
Ising spin systems, can have competing anisotropic magnetic interactions (see for example \cite{GM87, GLL11}). 
For example, ferromagnetic Ising models can have anisotropic interactions that are weighted towards one of the principle lattice axes.

Mathematically, it is natural to consider anisotropic LD models because of the richer interaction between the features (i) and (ii).
The mass scaling feature (ii)  prevails and, hence, in Theorem~\ref{thm:exist_nonexist} of Section~\ref{sec:exist_nonexist} we readily establish  existence for small mass and nonexistence for large mass in direct analogy to parts 
(G\ref{item:G1}) and (G\ref{item:G3}) above. Our proof combines several techniques used in the literature in a novel way.
However, feature (i), wherein the ball is naturally replaced by the Wulff shape associated with the anisotropy,   is {\it subtle}: While the Wulff shape is minimal for the perimeter term, its relation to the second term is, in certain cases, unclear. 
Thus what is fundamentally different for these anisotropic LD models is the structure of minimizers for small mass regime.  As we show in this article, this question is rich and, indeed, our work opens up far more questions than it solves.  We now present and discuss our results for each class of models, and  in doing so, explicitly state the main theorems of this article.

\subsection{Class I: Anisotropy in the Surface Energy}

Consider a {\it surface tension} \[ f:\R^n\to[0,\infty)\] to be positively one-homogeneous, convex, and positive on $\R^{n}\setminus\{0\}$. For a set of finite perimeter $E$, we let
\[
\F(E):=\int_{\pa^*E}f(\nu_E)\,d\H^{n-1}\,
\]

be the associated anisotropic surface energy, where 
$\pa^*E$ denotes the reduced boundary, $ \nu_E$ the measure-theoretic outer unit normal, and 
$\H^{n-1}$ the $(n-1)$ dimensional Hausdorff measure. 
Our first class of anisotropic LD models is given by 
\begin{equation}\label{eqn: aniso drop}
\inf\, \Big\{\mathcal{E}_f(E):= \F(E)+ \V(E) \,\,  \big\vert\,    \ |E| =m \Big\}, 
\end{equation}
where $\V$ is defined as in  (\ref{eqn: iso drop}). When $f (\cdot) \, =  | \cdot|$ (the Euclidean norm), our problem (\ref{eqn: aniso drop}) reduces to the LD problem (\ref{eqn: iso drop}). 

Let us recall that the global minimizer of the anisotropic isoperimetric problem 
\[
	 \inf\big\{ \F(E)  \, \big\vert\,   |E| = m\big\}
\] 
is (a dilation or translation of) the {\it Wulff shape} $K$ of $f$ where 
\begin{equation}\label{Wulffshape}
  K\, : =\, \bigcap_{\nu\in\S^{n-1}}\big\{x\in\R^{n}  \, \big\vert\,  x\cdot\nu<f(\nu)\big\}\,;
\end{equation}
see \cite{brothersmorgan,fonseca_wulff_rev,fonsecamuller_wulff,taylor_roma_73,taylorstanford75}. 
The Wulff shape is a bounded convex set that plays the role of the {\it ball} in the anisotropic setting.

Given that the anisotropic perimeter dominates the nonlocal term for small mass, 
one might initially be tempted to infer that, in this regime, the perimeter completely determines the behavior of minimizers, and minimizers are 
the Wulff shape for small mass. On the other hand, this would require Wulff shapes to be critical points of the {\it isotropic} repulsion. It turns out that the Euler-Lagrange equation serves as an obstruction for the minimality of the Wulff shape when $f$ satisfies certain smoothness and ellipticity conditions, but not in general: the Wulff shape can be a minimizer of \eqref{eqn: aniso drop}.
As we show, the role of the Wulff shape  {\it depends crucially on the regularity and ellipticity of $f$.} 
 To this end, let us introduce two important classes of surface tensions. 
We say that $f$ is a {\it smooth elliptic surface tension} if $f \in C^\infty(\R^n \setminus\{0\})$ and there exist constants $0<\l\leq\Lambda<\infty$ such that for every $\nu\in\S^{n-1}$,
\[
  \l\,|\tau|^2\leq \nabla^2f(\nu)[\tau,\tau]\leq\Lambda\,|\tau|^2\,
\]
for all $\tau \in \R^n$ with $\tau\cdot \nu =0.$
For such surface tensions, the corresponding Wulff shape has $C^\infty$ boundary and is uniformly convex.
 We say that $f$ is a {\it crystalline} surface tension if for some $N$ finite and $x_i \in \R^n,$
\[
f(\nu) = \max_{1\leq i\leq N} x_i \cdot \nu.
\]
For crystalline surface tensions, the corresponding Wulff shape $K$ is a convex polyhedron.


\eject
One of the main contributions of this paper is the following two contrasting Theorems: one general result for 
smooth surface elliptic tensions $f$ and  a sharply contrasting  example for a crystalline surface tension in 2D. 
\begin{theorem}\label{thm: not wulff}
Let $n \geq 2$ and $m>0$. Let $f$ be a smooth elliptic surface tension with Wulff shape $K$. Then we have the following two statements. 
\begin{itemize}
\item[(i)] 	Suppose  \[ \alpha \in \left(0,n-\frac{1}{3}\right).\]
 Then   $K$ is a critical point of \eqref{eqn: aniso drop} if and only if $f$ is the Euclidean norm. 
\vspace{0.1cm}
\item[(ii)] Suppose \[ \alpha \in \left(0,n- (\sqrt{2}-1)\right).\]
 Then there exists $\bar m_1$ depending on $n, f$, and $\alpha$ such that the following holds. Suppose $E$ is a minimizer of \eqref{eqn: aniso drop} for mass $m \leq \bar m_1$. Then for no other mass $m$ a dilation of $E$ is even a critical point, unless $f$ is the Euclidean norm and $E$ is a ball.
\end{itemize}
\end{theorem}
When surface tensions (e.g. crystalline) lack these smoothness and ellipticity properties,  it no longer makes sense to write down the Euler-Lagrange equation of \eqref{eqn: aniso drop} for arbitrary (smooth, compactly supported) variations. This means that the analysis for Theorem \ref{thm: not wulff} 
cannot be extended to this case. However, this is not purely a technical issue: Indeed, in contrast to the smooth elliptic case, we have the following result in the crystalline case.
\begin{theorem}\label{thm: crystals} 
Let $n=2$,  $\alpha \in (0,2)$, and let $f$ be the surface tension 
\begin{equation}
	\label{eqn: f}
	f(\nu) = \frac{1}{2} \| \nu\|_{\ell^1(\R^2)}=\frac{1}{2} (|\nu\cdot e_1|+|\nu \cdot e_2|),
\end{equation}
 whose corresponding Wulff shape $K$ is the square $[-1/2,1/2]\times[-1/2,1/2]$ of volume one.
 There exists $\bar m_2$ depending on $\alpha$ such that for $m\leq \bar m_2$ the Wulff shape is the unique (modulo translations) minimizer of \eqref{eqn: aniso drop}.
\end{theorem}

Let us provide a few comments on these theorems and their proofs. Theorem~\ref{thm: not wulff} sheds considerable light on the case (G\ref{item:G2}), which states that balls minimize \eqref{eqn: iso drop} for small masses. It is tempting to interpret this result as a consequence of scaling: for sufficiently small mass, the perimeter term dominates the nonlocal term and completely governs the behavior of minimizers, and hence the minimizers are balls. Theorem~\ref{thm: not wulff}(i) shows that this is not the case, since the energies in \eqref{eqn: iso drop} and \eqref{eqn: aniso drop} scale in the same way. 
Moreover,  in contrast to the classical liquid drop model where for every mass below a certain threshold, the minimizer of \eqref{eqn: iso drop} is just a dilation of the same set (i.e., the ball), if $E$ is a minimizer of \eqref{eqn: aniso drop} for suitably small mass, then a dilation of $E$ cannot be critical for any other mass.
A natural question then remains as to the nature of minimizers. To this end, we do present some partial results in 
Theorem \ref{thm: regularity} by showing that a minimizer is a small uniformly convex perturbation of the Wulff shape. A natural question then remains as to the nature of minimizers. After rescaling, the boundaries of minimizers converge in the Hausdorff topology to the boundary of the Wulff shape as $m\to 0$. For smooth and elliptic surface tensions, the regularity theory then implies that rescaled minimizers converge smoothly to the Wulff shape (in particular, they are uniformly convex for $m$ sufficiently small).  See Section 2 for more details. It is not clear if one could expect to give an explicit characterization of minimizers.


The proof of Theorem~\ref{thm: not wulff} is based upon an analysis of the first variation of \eqref{eqn: aniso drop}.
Besides regularity of minimizers (established in Theorem \ref{thm: regularity}) and a characterization of sets with constant first variation of $\F$ in \cite{AnisoAlex}, the main tool needed to prove 
Theorem ~\ref{thm: not wulff} is the fact that only balls have constant first variation of $\V$. While this was known for $\alpha \in (0,n-1)$, we produce more delicate arguments to extend the result to $\alpha \in [n-1, n - \frac{1}{3})$. To this end, we use a moving planes argument to prove the following:  


\begin{theorem}\label{thm: Riesz char}
 Fix $n\geq 2$ and $\alpha \in (0, n-1/3)$. Suppose that $E\subset \R^n$ is a bounded domain with $\pa E$ of class
\begin{align*}
 C^1\ \;  & \text{ if } \alpha<n-1\, ,\\
 C^{1, \gamma} & \text{ with } 1 + \gamma> 1/(n-\alpha) \text{ if } \alpha \in [n-1, n-1/2)\,,\\
 C^{2,\gamma} & \text{ with } 2+ \gamma > 1/(n-\alpha)\text{ if } \alpha \in [n-1/2, n-1/3)\,.
\end{align*}
Let $v_E(x)$ be the Riesz potential \[
 v_E(x)= \int_{E} \frac{dy}{|x-y|^\alpha }.
\]
 If $v_E$ is constant on $\pa E$, then $E$ is a ball.
\end{theorem}

Our Riesz potential restrictions  $\alpha \in (0, n-1/3)$ in Theorem \ref{thm: Riesz char} and  Theorem \ref{thm: not wulff} (i), as well as the requirement $\alpha \in \left(0,n- (\sqrt{2}-1)\right)$ for Theorem \ref{thm: not wulff} (ii), warrant the following remark. 

\bigskip

\begin{remark}[\emph{The range of Riesz potentials}]\label{remark-0}

Theorem~\ref{thm: Riesz char} was established for the Coulombic case $\alpha=n-2$ in \cite{Fraenkel00} and was extended to $\alpha \in (0,n-1)$ in \cite{LuZhu12}, both using the method of moving planes; see also \cite{Reichel09}. 
The case when $\alpha \geq n-1$ is significantly more delicate, principally due to the fact that the Riesz potential $v_E$ is merely H\"{o}lder continuous in this case; see \eqref{eqn: Holder}.

Our proof of Theorem~\ref{thm: Riesz char} in the subtler case $\alpha \in [n-1,n-1/3)$ pairs the method of moving planes on integral forms in the spirit of \cite{ChLiOu2006,LuZhu12} with some new reflection arguments and estimates on how the Riesz potential grows compared to its reflection across a hyperplane.
 

In order to apply 
Theorem~\ref{thm: Riesz char} to Theorem \ref{thm: not wulff} (i) and (ii), we need to establish, respectively,  regularity for the 
Wulff shape $K$ and for a minimizer $E$ of  \eqref{eqn: aniso drop} for mass $m$. The regularity of the Wulff shape $K$ depends on that of the elliptic surface tension which we have conveniently assumed to be $C^\infty$. With this smoothness assumption, we have sufficient regularity to directly employ Theorem~\ref{thm: Riesz char}. On the other hand, the regularity result for minimizers of $E$ (Theorem \ref{thm: regularity}) gives a further restriction on $\alpha$, yielding the assumption  $\alpha \in \left(0,n- (\sqrt{2}-1)\right)$.

After the submission of this article, G\'{o}mez-Serrano, Park, Shi, and Yao extended Theorem~\ref{thm: Riesz char} in \cite[Theorem C]{GPSY19} to the full range $\alpha \in (0,n)$ using continuous Steiner symmetrization. Furthermore, their proof of Theorem~\ref{thm: Riesz char} applies to sets with Lipschitz regular boundaries. With this result in hand, one can remove the technical assumptions from Theorem~\ref{thm: not wulff}(i) and (ii) to extend the results to all $\alpha \in (0,n).$
\end{remark}

\bigskip

Our Theorem \ref{thm: crystals} is in contrast to the 
 smooth elliptic setting, and together with Theorem~\ref{thm: not wulff} 
demonstrates an interesting situation where the regularity and ellipticity of the surface tension govern a fundamental aspect of the problem: whether the isoperimetric set is the minimizer of \eqref{eqn: aniso drop}.
Typically in anisotropic isoperimetric problems, the regularity and ellipticity of the surface tension affect quantitative aspects of the problem (for instance, regularity of quasi-minimizers), but not qualitative aspects of the problem. 
Theorem \ref{thm: crystals} should be regarded as an example (or counter-example), and not generic for crystalline surface tensions.  It is crucially based upon a 2D result of Figalli and Maggi (cf. Theorem \ref{thm:cryst_min}) which proves that quasi-minimizers of crystalline anisotropic surface energies must be convex polygons. Minimizers of \eqref{eqn: aniso drop} are quasi-minimizers of the anisotropic perimeter, so this effectively transforms \eqref{eqn: aniso drop} to a finite dimensional problem. For the simple case of a square Wulff shape, one can explicitly calculate $\V$. 
While we believe the result holds true in 2D for Wulff shapes that are regular polygons, our calculation uses the symmetries of the Wulff shape given by the dihedral group of order 8. Thanks to a recent result by Figalli and Zhang posted after the submission of this article (see \cite[Theorem 1.1]{FiZh2019}), it is also possible to extend our Theorem \ref{thm: crystals} to higher dimensions where the Wulff shape is given by a cube (cf. Remark \ref{rem:higher D}).

We remark that the subtleties of addressing \eqref{eqn: aniso drop} for crystalline surface tensions 
 highlights the  lack of a general theory in the modern calculus of variations to address extremal 
notions, like  criticality,  for nondifferentiable, nonconvex functionals. A key point becomes understanding among what variations one can compute the Euler-Lagrange equation and whether one can derive meaningful information from computing first and second variations among a restricted class of variations. This question is an important one in the setting of crystalline mean curvature flow; see, among others, \cite{bell04,BeNoPa2001_1,BeNoPa2001_2,bellettininovagariey,ChaMorPon15,ChMoNoPo17,taylor78,taylor_1993} and references therein. The paper \cite{DMMN} also investigates this theme.

\smallskip

\subsection{Class II: Anisotropy in the Surface Energy and the Repulsive Term}

In light of Theorem \ref{thm: not wulff} and the physical motivation, it would seem natural to replace the repulsive term $\V(E)$ with $f$-driven anisotropic interactions which are maximized (under fixed volume) by the Wulff shape $K$. 
To this end, let us assume the surface tension  $f$ is smooth and elliptic and denote by 
 $f_*$  the dual to $f$ defined by 
	\[
		f_*(x) = \sup \{ x\cdot \nu \, \big\vert \,  f(\nu )\leq 1\}\,.
	\]
Note that the Wulff shape $K$ can be equivalently expressed as the unit ball for $f_*$, that is, $K = \{ x: f_*(x) <1\}.$
We consider three classes of variational problems in the spirit of \eqref{eqn: aniso drop}:
\begin{equation}\label{eqn: general}
\inf\Big\{
\F(E) + \mathcal{U}_1(E) \, \big\vert \,    |E| =m\Big\}\,
\end{equation}
and
\begin{equation}\label{eqn: general confined}
\inf\Big\{
\F(E) + \mathcal{U}_i(E) \, \big\vert \,   |E| =m, \ E \subset B_{c_{n,f}m^{1/n}}\Big\}\,, \qquad i =2,3,
\end{equation}
for some $c_{n,f}$ depending on $n$ and $f$, where we let
 \begin{equation}\label{eqn: functionals} \begin{split}
 \mathcal{U}_1(E) &:= \sup_{y \in \R^n} \int_E f_*(x-y)^{-\alpha} \, dx \qquad \text{ for } \alpha \in (0,1)\,\\
 \mathcal{U}_2(E) &:= - \inf_{y\in \R^n}\int_E f_*(x-y)^{\beta} \, dx \qquad \text{ for } \beta \in (0,\infty)\,,\\
 \mathcal{U}_3(E) &:=- \inf_{y \in \R^n } \int_E \log (f_*(x-y)) \, dx\,.
 \end{split}
 \end{equation}
The confinement constraint in \eqref{eqn: general confined} is needed, otherwise the infimum is minus infinity for all $m$ and so no minimizer exists.
As we show in Section \ref{sec:doubly_aniso},  each $\mathcal{U}_i(E)$ in \eqref{eqn: functionals} is maximized by the Wulff shape among sets of a fixed volume, so the variational problems \eqref{eqn: general} and \eqref{eqn: general confined} exhibit  {\it both} of the analogous two features of the isotropic LD model \eqref{eqn: iso drop}.  In Section \ref{sec:doubly_aniso} we prove the following theorem.

\begin{theorem}\label{thm: wulff min} Let $n\geq 2$ and let $f$ be a smooth elliptic surface tension. There exists a constant $m'=m'(n, f, \mathcal{U}_i)$ such if $m \leq m'$, any minimizer of \eqref{eqn: general} or \eqref{eqn: general confined} is a Wulff shape. 
\end{theorem}

The main tool in Theorem~\ref{thm: wulff min} is a strong form of the quantitative Wulff inequality from \cite{Neum16}.
Our method of proof is flexible; the two key ingredients are the subcritical scaling of $\mathcal{U}_i$ with respect to the surface energy and the criticality of $K$, and one can adapt the proof to other functionals satisfying these properties.

\begin{remark}[\emph{Equilibrium Figures \`{a} la Poincar\'{e} with an Anisotropic Potential}]\label{remark-1}
Perhaps the most natural way to incorporate anisotropic repulsions would be to replace $\V(E)$ with 
\begin{equation}\label{eq-aniso-pot}
\V_f(E) = \int_E \int_E \frac{1}{f_*(x-y)^{\alpha}} \, dx dy\, ,
\end{equation}
and consider the minimization problem 
\begin{equation}\label{eqn: double anisotropic}
\inf\, \Big\{
\F(E) + \V_f(E) \, \big\vert \,   |E| =m\Big\}.
\end{equation}
Apart from the trivial case where  $ f(\nu) = | A \nu|$  for a positive definite matrix $A$ {\rm (}where a linear change of variables transforms \eqref{eqn: double anisotropic} into the isotropic LD model \eqref{eqn: iso drop}{\rm )}, we do not know whether the Wulff shape is the minimizer for \eqref{eqn: double anisotropic} for small mass.

The issue is related to the anisotropic problem 
\begin{equation}\label{anisotropic capacity problem}
  \inf \, \Big\{ -\V_f(E) \, \big\vert \,   |E| =m\Big\} .
 \end{equation}
 In the isotropic case ($f$ being the Euclidean norm) 
 and $\alpha = n-2$ (Newtonian), the problem \eqref{anisotropic capacity problem} 
   has a long history and was made famous by Poincar\'{e} in 
his 1902 treatise {\rm \cite{Poincare2}} on equilbrium figures.  If the total angular momentum vanishes then the ``problem of the equilibrium figure" reduces to mimimization of  $- \V$  (with Newtonian $\alpha = n-2$) with its attractive gravitational potential. Poincar\'{e} 
 claimed the unique solution was the ball. A complete solution to the problem had to wait almost a century with the work of Lieb  {\rm \cite{Li2}} in 1977  whose proof was based 
on the strict Riesz rearrangement inequality.

The problem for equilibrium figures of anisotropic potentials as solutions of \eqref{anisotropic capacity problem} remains an important open problem for general $f$.
Unfortunately, for general $f$,  the  Riesz rearrangement inequality fails to hold true 
{\rm (cf. \cite{VANSCHAFTINGEN2006539})}. To our knowledge,  even the criticality of the Wulff shape for \eqref{anisotropic capacity problem} is unclear.
\end{remark}

\bigskip
\noindent{\bf Outline of the article.} 
 
\noindent  -- In  Section \ref{sec: prelim}, we present some preliminary facts about the variational problem \eqref{eqn: aniso drop} and deduce some a priori structure and regularity properties of minimizers, provided they exist. 
\medbreak

\noindent -- In Section \ref{sec:exist_nonexist}, 
we prove that (\ref{eqn: aniso drop}) admits a minimizer when $m$ is sufficiently small, while no minimizer exists when $m$ is sufficiently large.\medbreak

\noindent -- In Section \ref{sec:smooth_aniso}, we establish Theorem~\ref{thm: not wulff}, first proving Theorem \ref{thm: Riesz char}. 
\medbreak
 
\noindent  -- In Section \ref{sec:cryst_aniso} ,  we prove Theorem \ref{thm: wulff min}.
\medbreak

\noindent -- In Section \ref{sec:doubly_aniso}, we address the inclusion of anisotropic potentials, proving 
Theorem \ref{thm: wulff min}. 
\medbreak

\noindent -- In Section \ref{open}, we conclude by noting and recalling some open problems.
 \medbreak
 
\noindent --We finally include a few details in an appendix.

%
%
%
%
%


\section{Preliminaries} \label{sec: prelim}

Throughout the paper we denote constants that might change from line to line by $C$ and keep track which parameters these constants depend on in parentheses. For specific constants that reappear elsewhere in the paper we use lower-case letters and denote their dependencies with subscripts such as $c_{n,f}$, $c_{n,\alpha}$.

  \subsection{Basic properties of the surface energy}\label{subsec: wulff}
  
The fact that  $\F$ is uniquely minimized by translations of the {\it Wulff shape $K$ of $f$} defined in (\ref{Wulffshape}) can be restated in a scaling invariant way via the Wulff inequality
\begin{equation}\label{eqn: wulff inequality}
\F(E) \geq n|E|^{(n-1)/n}|K|^{1/n}.
\end{equation}
Throughout, we denote the dilation of the Wulff shape by $K_\rho = \rho K$.
Recall that  $f_*$ denotes the Fenchel dual of $f$ defined by 
\[
f_*(x) = \sup \{ x\cdot \nu : f(\nu )\leq 1\}, 
\]
with the Wulff shape $K$ equivalently expressed as the unit ball for $f_*$. 
Let 
\beqn \label{eqn:l and L}
\ell_f = \inf \{ f(\nu)  \, \big\vert\,  \nu\in \S^{n-1}\}, \qquad L_f = \sup \{ f(\nu)  \, \big\vert\,  \nu \in \S^{n-1}\}.
\eeqn
It follows that
\[
\frac{1}{L_f} = \inf \{ f_*(x)  \, \big\vert\,  x \in \S^{n-1} \} , \qquad \frac{1}{\ell_f} = \sup \{f_*(x)  \, \big\vert\,  x \in \S^{n-1}\}\,.
\]
In particular, we observe that $B_{\ell_f} \subset K \subset B_{L_f}$, where $B_r$ denotes the ball of radius $r$ in $\R^n$. 

 \subsection{Basic properties of the nonlocal energy}
Fix $\alpha \in (0,n)$ and let $v_E:\R^n\to \R$ denote the Riesz potential of $E$ given by
\begin{equation}\label{eqn: vE}
 v_E(x)= \int_{E} \frac{dy}{|x-y|^\alpha }\, .	
\end{equation}
In this way, $\V(E)= \int_E v_E(x) \,dx$.
Taking $r$ such that $\omega_n r^n=|E|$, where $\omega_n$ denotes the volume of the unit ball in $\R^n$, we have
\begin{equation}\label{eqn: bounded}
\| v_E\|_{L^\infty(\R^n)} \leq \| v_{B_r(0)}\|_{L^\infty(\R^n)} = v_{B_r(0)}(0) = \frac{n\om_n}{n-\alpha} r^{n-\alpha}\,.
\end{equation}
Furthermore, for $k=\lfloor n-\alpha \rfloor$ and $\beta\in(0,1)$ with $k+\beta <n-\alpha$, we have
\begin{equation}\label{eqn: Holder}
\|v_E\|_{C^{k,\beta}(\R^n)}\leq C(n,|E|, k, \beta)\,;
\end{equation}
see, e.g. \cite[Lemma 3]{Reichel09},\cite[Proposition 2.1]{BoCr14}.

The functional $\V(E)$ is Lipschitz continuous with respect to the symmetric difference in the sense that there exists $c_{n,\alpha}>0$ such that
\begin{equation}\label{eqn: lip}
|\V(E) - \V(F) | \leq c_{n,\alpha}\, m^{(n-\alpha)/n} |E\triangle F| \qquad \text{ for } |E| \leq |F|=m.
\end{equation}
Indeed, thanks to \eqref{eqn: bounded},
\begin{align*}
\V(E) - \V(F) &= \int_{\R^n} \int_{\R^n} \frac{\chi_E(x)\chi_E(y) - \chi_F(x)\chi_F(y)}{|x-y|^\alpha } \, dx dy\\
& = \int_{\R^n} v_E(y)(\chi_E(y) - \chi_F(y))\, dy + \int_{\R^n} v_F(x)(\chi_E(x) - \chi_F(x))\, dx\\
&\leq \frac{2n\om_n^{\alpha/n}}{n-\alpha} m^{(n-\alpha)/n} |E\triangle F|.
\end{align*}
Hence, \eqref{eqn: lip} holds with
	\beqn\label{eqn:cnalpha}
		c_{n,\alpha}:= \frac{2n\omega_n^{\alpha/n}}{n-\alpha} \, .
	\eeqn

\subsection{Scaling of the energy and  initial energy bounds}
Given a set of finite perimeter $E\subset \R^n$, note that 
 \begin{align*}
 \mathcal{E}_f(r E) 
 &= r^{n-1} \{ \F(E) + r^{n+1-\alpha} \V(E)\}.
 \end{align*}  
 Hence, setting $\e = m^{(n+1-\alpha)/n}$, $E_m$ is a minimizer of \eqref{eqn: aniso drop} with mass $m$ if and only if $E= m^{-1/n}E_m$ is a minimizer of the variational problem
 \begin{equation}\label{eqn: aniso drop 2}
 \inf\Big\{ \mathcal{E}_{\e,f} (E)\  \Big| \ |E| =1\Big\}, \quad
 	\text{where} \quad  \mathcal{E}_{\e, f} (E ) := \F(E) + \e\, \V(E).
 \end{equation}
It will often be convenient to consider minimizers of this rescaled problem in place of \eqref{eqn: aniso drop}. Let us give two initial bounds on the infimum in \eqref{eqn: aniso drop 2}, which here and in the sequel we denote by $\bar \E_{\e,f}$. The first bound is essentially optimal for small $\e$, whereas the second is essentially optimal for large $\e$. 
First, comparing to $K_r$ with $r=|K|^{-1/n}\leq 1/(\ell_f \om_n^{1/n})$, we find that
\begin{equation}\label{eqn: energy bound 1}
\begin{split}
	\bar \E_{\e,f} \leq \E_{\e,f}(K_r)& \leq n|K| r^{n-1} + C(n,\alpha) \e \\
	&\leq C(n,\alpha, \ell_f, L_f),
\end{split}
\end{equation}
where the third inequality holds only when $\e \leq 1.$
Next, fix $N\in \mathbb{N}$, let $\rho = (N|K|)^{-1/n}$, and consider the sequence $\{E_k\}$ given by $E_k = \bigcup_{j=1}^N (K_\rho + kje_1)$, so that
\[
\E_{\e,f}(E_k) = N^{1/n} |K|^{1/n} n + \e c_{n,\alpha} N^{(-n+\alpha)/n} + o_k(1).
\]
Optimizing in $N$ and letting $k\to \infty$, we find that 
\begin{equation}\label{eqn: energy bound 2}
\bar \E_{\e,f} \leq C(n,\alpha, \ell_f, L_f)\, \e^{1/(n+1-\alpha)}\,.
\end{equation}

\subsection{Quasi-minimality of minimizers}\label{subsec: q min} Let us recall two useful notions of sets that almost minimize the surface energy $\F$.
We say that $E$ is a $(\Lambda,r)$-quasi-minimizer of $\F$ if for any  $x\in\R^n$
\begin{equation}\label{eqn: qmin}
\F(E) \leq \F(F)+ \Lambda\,|E\triangle F| \qquad \text{ for all  $F$ with }\ \  F\triangle E  \cc B_r(x)\, .
\end{equation}
We say that $E$ is a $q$-volume-constrained quasi-minimizer if
\[
\F(E) \leq \F(F)+ q\,|E\triangle F| \qquad \text{ for all  $F$ with }\ \  |F|=|E| \,.
\]
Since the potential $\V(E)$ is Lipschitz we can deduce that any minimizer of \eqref{eqn: aniso drop 2} satisfies both of these  quasi-minimality properties:
\blemma \label{lem:qmin} Let $E$ be a minimizer of \eqref{eqn: aniso drop 2}. 
  Then 
  \begin{enumerate}
  	\item $E$ is a $c_{n,\alpha}\e$-volume-constrained quasi-minimizer of $\F$, with $c_{n,\alpha}>0$ given by \eqref{eqn:cnalpha}, and
  	\item $E$ is a $(\Lambda,1)$-quasi-minimizer of $\F$ for some $\Lambda>0$ which depends on $f,n,\alpha$, and $\e$, and is bounded independently of $\e$ for any $\e \leq 1.$
  \end{enumerate}
\elemma
The proof of Lemma~\ref{lem:qmin} is standard, but we include it in Appendix~\ref{app: proof of q min} for the convenience of the reader. A classical argument, see for instance \cite[Theorem 21.11]{Maggi2012}, shows that $(\Lambda,r_0)$-quasi-minimizers satisfy uniform density estimates: provided $\Lambda r_0 /\ell_f \leq 1$,
if $x \in \pa E$ and $r<r_0$, then 
	\begin{equation}\label{eqn: density ests}
c_0 \leq \frac{|E \cap B_r(x)|}{\om_nr^n} \leq 1-c_0 \qquad \text{ with } \quad c_0:=\frac{\ell_f^n}{4^nL_f^n}.
\end{equation}
%

Recall that a set $E$ is \emph{$\F$-indecomposable} if whenever $E=E_1\cup E_2$ with $E_1, E_2$ disjoint and $\F(E)=\F(E_1) + \F(E_2)$, then $|E_1||E_2|=0$. If a $\F$-indecomposable set $E$ with $|E|\leq 1$ satisfies the lower density estimates of \eqref{eqn: density ests} up to scale $r_0$, then
\[ 
\diam E \leq 2^{n+1} (c_0\om_n)^{-1} \, r_0^{1-n}.
\]
Indeed, let $d =\diam E$ and  take $N = \lfloor d r_0^{-1}\rfloor $ points $\{x_i\}_{i=1}^N \subset \pa E$ such that $\{B_{r_0/2}(x_i)\}_{i=1}^N$ are pairwise disjoint. Then
\[
1 = |E| \geq \sum_{i=1}^N |E\cap B_{r_0/2}(x_i)| \geq 2^{-n} N c_0 \om_n r_0^n \geq  2^{-n-1} c_0 \om_n d r_0^{n-1}.
\]  

Observe that minimizer $E$ of \eqref{eqn: aniso drop 2} is indecomposable. Suppose $E=E_1 \cup E_2$ for disjoint sets $E_1, E_2$ with $\F(E)=\F(E_1) + \F(E_2)$. Applying the diameter bound to each $\F$ indecomposable component of $E_1, E_2,$ we find that $E_1$ and $E_2 + ke_1$ are disjoint for $k$ sufficiently large, so taking $G = E_1 \cup( E_2 + ke_1)$, we have $\E_{\e,f}(G)\leq\E_{\e,f}(E)$, with strict inequality unless $|E_1||E_2|=0$.

Furthermore, note that for minimizers $E_\e\subset \R^n$ of the rescaled problem \eqref{eqn: aniso drop 2}, using the estimate \eqref{eqn: lip}, we have that $\F(E_\e) \leq \F(K) + C(n,\alpha)\,\eps$. This implies that $E_\e \to K$ in $L^1(\R^n)$ (up to translation). Furthermore, the density estimates \eqref{eqn: density ests} paired with the $L^1$ convergence yield $d_H(\pa E_\e, \pa K)\to 0$ as $\e \to 0$, where $d_H$ denotes the Hausdorff distance.

\subsection{Regularity of minimizers}
Suppose $f$ is a smooth elliptic surface tension. The first variation of $\F(E)$ with respect to a variation generated by $X\in C^{1}_c(\R^n,\R^n)$ is given by 
\[
\delta \F(E)[X] = \int_{\pa^*E} \DIVV^{\pa^* E}\big(\na f\circ \nu_E\big) X\cdot\nu_E \, d\H^{n-1}\,
\]
where $\Div^{\pa^* E}$ denotes the tangential divergence along $\pa^* E$, and $\pt^*$ denotes the reduced boundary. We define $H_E^f:\pa^* E\to \R$ by $H_E^f= \DIVV^{\pa^* E}\big(\na f\circ \nu_E\big)$. Often $H_E^f$ is called anisotropic mean curvature in analogy with the isotropic setting.
The first variation of $\mathcal{V}(E)$ with respect to a variation generated by $X\in C^{1}_c(\R^n,\R^n)$ is given by 
\[
\delta \V(E)[X] = \int_{\pa^*E} v_E(x) X\cdot\nu_E \, d\Hi^{n-1}\,,
\]
with $v_E(x)$ as defined in \eqref{eqn: vE}.

We say that a set $E$ is a  critical point of \eqref{eqn: aniso drop} if $\delta(\F(E) + \V(E))[X]=0$ for all variations with $\int_{\pa^*E} X\cdot \nu_E\, d \Hi^{n-1}=0$, i.e. variations that preserve volume to first order. Hence, a volume-constrained critical point $E$ of \eqref{eqn: aniso drop} satisfies the Euler-Lagrange equation
\begin{equation}\label{eqn: EL}
H_E^f(x) + v_E(x) = \mu  \qquad \text{ for }x \in \pa^* E\,.
\end{equation}
Here, $\mu$ is a Lagrange multiplier coming from the volume constraint.

We have the following regularity properties of minimizers of \eqref{eqn: aniso drop}. 

\begin{theorem}\label{thm: regularity}
  Fix $n\geq 2$ and $\alpha \in (0,n)$. Suppose 
  $f$ is a smooth elliptic surface tension and let $E$ be a minimizer of \eqref{eqn: aniso drop} for mass $m$. 
  \begin{myenumerate}
  	\item The reduced boundary $\pa^*E$ is a $C^{2,\beta}$ hypersurface for all $\beta<\beta_0:=\min\{1,n-\alpha\}$.
  	\end{myenumerate}
  	Furthermore, there exist $m_4$ and $m_5$ depending on $n,f,$ and $\alpha$ with $m_4 \geq m_5$ such that the following statements hold. 
  	\begin{myenumerate}
  	\item If $m \leq m_4$, then $\pa E$ is a $C^{2,\beta}$ hypersurface for all $\beta \in (0,\beta_0)$, and in fact can locally be written as a small $C^{2,\beta}$ graph over the boundary of the Wulff shape of mass $m$. 
  	\item If $m \leq m_5$, then $E$ is uniformly convex.
  \end{myenumerate} 
  \end{theorem}
   Theorem~\ref{thm: regularity} can be deduced from known arguments and regularity results in the literature. We outline the proof in Appendix~\ref{appendix: regularity} for completeness, and following the proof we make several remarks about sharper forms of Theorem~\ref{thm: regularity} that can be proven but are not needed in this paper.

 Theorem~\ref{thm: regularity} can be deduced from known arguments and regularity results for quasi-minimizers in the literature (see e.g. \cite{alm66, SSA77, bomb82, DuzaarSteffen02}). Likewise sharper forms of Theorem~\ref{thm: regularity} can be proven but are not needed in this paper. However, we remark on these extensions here for future reference.

\section{Existence and nonexistence of minimizers} \label{sec:exist_nonexist}

In this section we prove that the energy $\varE_f$ admits a minimizer when $m$ is sufficiently small, while no minimizer exists when $m$ is sufficiently large.
\begin{theorem}\label{thm:exist_nonexist} Let $n \geq 2$, and let $f$ be a surface tension with $\ell_f$ and $L_f$ given by \eqref{eqn:l and L}.
\begin{enumerate} \addtolength{\itemsep}{6pt}
	\item For all $\alpha \in (0,n)$ there exists $m_1=m_1(\alpha,n, \ell_f, L_f)>0$ such that for all $m \leq m_1$, the variational problem \eqref{eqn: aniso drop} admits an essentially bounded and $\F$-indecomposable minimizer $E \subset \R^n$.
	\item For all $\alpha \in (0,2)$, there exists $m_0 = m_0 (\alpha, n, \ell_f, L_f)$ such that for all $m>m_0$, no minimizer exists in \eqref{eqn: aniso drop}.	
\end{enumerate}
\end{theorem}
It is not known whether minimizers of \eqref{eqn: aniso drop} exist for large masses, even in the isotropic case, for $\alpha \in [2,n).$ We prove Theorem~\ref{thm:exist_nonexist}(i) in the same way it was shown in the isotropic case in \cite[Theorem 2.2]{KM13} and \cite[Theorem 3.1]{KM14}; the details are given in Appendix~\ref{app: existence}.  The proof of Theorem~\ref{thm:exist_nonexist}(ii) combines an elegant argument of \cite{FrKiNa2016} with the diameter bound 
\begin{equation}\label{eqn: diameter bound}
	\diam E \leq C(n,\alpha, \ell_f,L_f)\, m
\end{equation}
for minimizer of \eqref{eqn: aniso drop} with mass $m$. This diameter estimate, shown in Appendix~\ref{app: existence}, was originally established in \cite[Lemma 7.2]{KM14} by showing that the lower density estimates of \eqref{eqn: density ests} hold up to an improved scale.

\medskip

\begin{proof}[Proof of Theorem \ref{thm:exist_nonexist}(i)]
It is equivalent to show that there is $\e_1 = \e_1(n,f,\alpha) \leq 1$ such that a minimizer exists for the rescaled problem \eqref{eqn: aniso drop 2} for $\e \leq \e_1$. Let $\{F_k\}$ be a minimizing sequence for \eqref{eqn: aniso drop 2} with $|F_k|=1$. Lemmas \ref{lem: nonopt} and \ref{lem: non opt satisfied} show that there exists $\bar \rho$ such that for each $k$, there exists $\rho_k \leq \bar \rho$ satifsying
\[
\E_{\e, f}(G_k)\leq \E_{\e, f}(F_k),
\]
where $G_k$ is a dilation of $F_k\cap B_{\rho_k}$ with $|G_k|=1$ and $G_k\subset B_{R_0}(0)$ with $R_0$ depending on $n,\alpha, \ell_f$, and $L_f$. Such a sequence, having $\F(G_k)\leq C$, is compact in the $L^1$ topology, so up to a subsequence, $G_k\to E$ in $L^1$ with $|E|=1$. The energy $\E_{\e,f}$ is lower semi-continuous with respect to $L^1$ convergence, so $E$ is a minimizer of \eqref{eqn: aniso drop 2}. The boundedness and indecomposability of a minimizer follow from the remarks in Section~\ref{subsec: q min}.
\end{proof}

\medskip

Before proving Theorem \ref{thm:exist_nonexist}(ii), let us fix some notation.
For fixed $t \in \R$ and $\nu\in \S^{n-1}$, we define the hyperplane
\[
H_{\nu, t} = \{ x\in\R^n  \, \big\vert\,  x\cdot \nu = t\}
\]
and the corresponding half-spaces
\[
H_{\nu, t}^{+} = \{ x\in\R^n  \, \big\vert\,  x\cdot \nu > t\}, \qquad H_{\nu, t}^{-} = \{ x\in\R^n  \, \big\vert\,  x\cdot \nu < t\}.
\]
For a given set $E$, we let $E_{\nu,t}^{\pm} = E \cap H_{\nu, t}^{\pm}$.

\medskip

\begin{proof}[Proof of Theorem \ref{thm:exist_nonexist}(ii)]
Suppose $E$ is a minimizer of \eqref{eqn: aniso drop} for mass $m$. For fixed $t \in \R$ and $\nu \in \S^{n-1}$, 
arguing as we did to show indecomposability in Section~\ref{subsec: q min}, 
we have
\begin{equation}\label{eqn: minimality}
\E_f(E ) 
 \leq \E_f(E_{\nu,t}^{+}) + \E_f(E_{\nu,t}^{-})\,.
\end{equation}
Next, recall that for any set of finite perimeter $E$ and $\nu \in \S^{n-1}$, we have 
	\[
	\F(E_{\nu,t}^+ ) + \F(E_{\nu, t}^-) = \F(E) + \{f(\nu) + f(-\nu)\} \Hi^{n-1}(E \cap H_{\nu, t})
	\]
	for a.e. $t \in \R$ (see \cite[Theorem 16.3 and Proposition 2.16]{Maggi2012}). 
So, rearranging \eqref{eqn: minimality} yields
\begin{equation}\label{eqn: compare1}
\V(E) - \V(E_{\nu,t}^+) - \V(E_{\nu, t}^-) \leq  \{f(\nu) + f(-\nu)\} \Hi^{n-1}(E \cap H_{\nu, t})
\end{equation}
for a.e. $t$.  We will integrate both sides of \eqref{eqn: compare1} with respect to $t \in \R$ and $\nu \in \S^{n-1}$. Integrating the right-hand side yields $2\F(B_1) m$. For the left-hand side, we observe that
\[
		\V(E) - \V(E_{\nu,t}^+) - \V(E_{\nu, t}^-) = 2 \int_{\R^n} \int_{\R^n} \frac{\chi_{E_{\nu,t}^+}(x) \chi_{E_{\nu,t}^-}(y)}{|x-y|^{\alpha}} \, dx dy
\]
and that
\[
\chi_{E_{\nu,t}^+}(x) \chi_{E_{\nu,t}^-}(y) = \chi_{\{x\cdot \nu > t > y\cdot \nu\}} (x,y) \chi_E(x) \chi_E(y).
\]
By the layer cake formula, 
\[
\int_{-\infty}^\infty \chi_{\{x\cdot \nu > t > y\cdot \nu\}} \, dt = \int_{-\infty}^{\infty} \chi_{\{(x-y)\cdot \nu > s > 0\} }(x,y) \, ds = [(x-y)\cdot\nu]_+\, .
\] 
So, by Fubini's theorem, integrating the left-hand side of \eqref{eqn: compare1} with respect to $t \in \R$ gives us
\begin{equation}\label{eqn: int over ell}
2\int_E \int_E \frac{[(x-y)\cdot\nu]_+}{|x-y|^{\alpha}} \, dx dy 
\end{equation}
We now integrate \eqref{eqn: int over ell} over $\nu \in \S^{n-1}$. Note that 
\begin{align*}
\int_{\S^{n-1}} [a\cdot\nu]_+ \,d\Hn &= |a| \int_{\S^{n-1}} [a/|a| \cdot \nu]_+ \,d\Hn= |a| \int_0^{\pi/2} \cos(\theta) \H^{n-2}(\S^{n-2}_{\theta}) \, d\theta \\&= (n-1)\om_{n-1}|a| \int_0^{\pi/2} \cos(\theta)\sin^{n-2}(\theta) \, d\theta = |a| \om_{n-1}.
\end{align*}
So, again using Fubini's theorem, integrating both sides of \eqref{eqn: compare1} over $t \in \R$ and $\nu \in \S^{n-1}$ yields
\beqn \label{eqn: nonlocal vs per}
2\om_{n-1} \int_E \int_{E} \frac{1}{|x-y|^{\alpha-1} }\,dx dy \leq 2\F(B_1) m.
\eeqn

When $\alpha \in (0,1],$ the left-hand side is minimized by $B_{(m/\omega_n)^{1/n}}$ and hence is bounded below by $C(n,\alpha)\,m^{2 - (\alpha-1)/n}$ for some constant $C(n,\alpha)>0$. It follows that  the existence of a minimizer implies that
\begin{equation}\label{eqn: mass bound}
m \leq C(n,\alpha,\ell_f, L_f).
\end{equation}
On the other hand, when $\alpha \in (1,2)$, the left-hand side of \eqref{eqn: nonlocal vs per} is bounded below by $2\om_{n-1}(\diam E)^{1-\alpha} m^2$. Rearranging this gives us the diameter lower bound
\[
m^{1/(\alpha-1)} \leq C(n,\alpha, \ell_f, L_f) \diam E
\]
 Pairing this with \eqref{eqn: diameter bound}, we establish \eqref{eqn: mass bound} in this case as well.
\end{proof}

\section{Smooth elliptic surface tensions: the Proof of Theorems \ref{thm: not wulff} and \ref{thm: Riesz char}}\label{sec:smooth_aniso}

In this section, we establish Theorem~\ref{thm: not wulff}. The main tool is Theorem \ref{thm: Riesz char}  
 which states that, under suitable regularity assumptions, balls are the {\it only} bounded sets for which the first variation of the nonlocal energy $\V(E)$ is constant. We first prove  Theorem \ref{thm: Riesz char} which, as described in Remark \ref{remark-0}, which extends previous results for $0< \alpha < n-1$ to cases where the Riesz potential lacks $C^1$ regularity. 
   
 \begin{proof}[Proof of Theorem~\ref{thm: Riesz char}]

Let us introduce some notation. For simplicity, we let $v = v_E$ throughout the proof.  Following the notation in Section~\ref{sec:exist_nonexist}, for any $t\in \R$, let 
\[
H_t = \{ x\cdot e_1 = t\}, \qquad H_t^- = \{ x\cdot e_1 < t\}, \qquad H_t^+ = \{x\cdot e_1>t\}.
\]
Let
\[
\Sigma_t= E\cap H_t^-
\]
and, defining $x^t = (2t-x_1, x_2,\dots, x_n)$ to be the reflection of $x$ across $H_t$, let 
\[
 \Sigma_t' = \{ x^t  \, \big\vert\,  x \in \Sigma_{t}\}.
\]
Finally, let  $v_t(x) = v(x^t)$ and set $G_t = E \setminus \overline{(\Sigma_t \cup \Sigma_t')}.$ Observe that $G_t\subset H_t^+$.\\

Since $E$ is bounded, $\Sigma_t$ is empty for $t$ sufficiently small and contains $E$ for $t$ sufficiently large. This means that 
 \[
\lambda = \sup \{ t  \, \big\vert\,  \Sigma_t' \subset E\}
\]
is finite. Furthermore, we have either 
	\smallskip
\begin{enumerate}[\Case 1.] 
	\item $\pa \Sigma_\lambda'$ is tangent to $\pa E$ for some point $\bar x \in \pa E \setminus H_\lambda$, or
	\item $H_\lambda$ is orthogonal to the tangent plane of $\pa E$ at some point $\bar x \in \pa E \cap  H_\lambda$.
\end{enumerate}
	\smallskip
We will show that in either case, $G_\lambda$ is empty and thus $E$ is symmetric across $H_\lambda$. 

\bigskip

For any $x \in \R^n$ we may write
\begin{align*}
v(x) & = \int_{\Sigma_\lambda} |x-y|^{-\alpha} \, dy + \int_{\Sigma_\lambda'} |x-y|^{-\alpha} \, dy + \int_{G_\lambda} |x-y|^{-\alpha}\,dy,\\
v_\lambda(x) & =  \int_{\Sigma_\lambda} |x-y|^{-\alpha} \, dy+ \int_{\Sigma_\lambda'} |x-y|^{-\alpha} \, dy + \int_{G_\lambda} |x^\lambda-y|^{-\alpha}\,dy
\end{align*}
and therefore we have 
\begin{equation}\label{eqn: reflection}
 v_\lambda(x) -v(x) = \int_{G_\lambda} \Big( |x^\lambda-y|^{-\alpha} - |x-y|^{-\alpha} \Big)\, dy.
\end{equation}
In particular, for any $x \in H_\lambda^{-}$, the right-hand side is strictly positive unless $G_\lambda$ is empty.

\bigskip

Let us first suppose that Case 1 occurs. Since $\bar x \in \pa \Sigma_\lambda' \cap \pa E$ and $\bar x \not\in H_\lambda$, it follows from construction that $\bar x $ is the reflection $x^\lambda$ for some $x\in \pa E \cap \pa \Sigma_\lambda.$ Since $x$ and $\bar x^\lambda$ both lie in $\pa E$, we have $v(\bar x) = v_\lambda(\bar x) = $ const. Hence \eqref{eqn: reflection} implies that $G_\lambda$ is empty and so $E$ is symmetric about $H_\lambda.$

\bigskip

Next, suppose that Case 2 occurs. Note that $e_1$ is parallel to the tangent plane of $\pa E$ at $\bar x$. Hence, up to a translation and a rotation of $E$ that fixes its orientation with respect to $e_1$, we may assume that $\bar x=0$ and the tangent plane to $\pa E$ at zero is the plane $\{x_n=0\}$. Moreover, we may locally express $\pa E$ as the graph over its tangent plane. More specifically, letting 
\[
x = (x',x_n) \quad \text{ and }  \quad B_r'(0)=\{|x'|<r\}\subset \R^{n-1},
\]
for $r$ suitably small, we may find a $C^{k,\gamma}$ function $\varphi: B_r'(0)\to \R$:
\begin{equation}\label{eqn: graph}	
\pa E \cap B_r(0)  = \left\{ \left(x', \varphi(x')\right): x' \in B_r'(0)\right\}.
\end{equation}

If $\alpha \in (0,n-1)$, we can argue exactly as in \cite{LuZhu12}. In this case, $v$ is differentiable (recall \eqref{eqn: Holder}) and thus $\pa_{e_1} v =0$. In particular, for $h>0$ small, we have 
\begin{equation}
	\label{eqn: claim0}
	{|v(\bar x+ h e_1) - v(\bar x- h e_1)| =  o(h).}
\end{equation}

\medskip

 Now suppose that $\alpha \in [n-1, n-1/3)$. For $h>0$ small, let $\varphi_h = \varphi( -h e_1)$
and consider the sequence $\{x_h\} \subset \pa E \cap H_\lambda^-$ defined by
\[
x_h = -h e_1 + \vphi_h e_n .
\]
 In this way, 
$x_h^\lambda = h e_1 + \vphi_h e_n \in H_\lambda^+$.
We claim that 
\begin{equation}\label{eqn: claim1}
|v(x_h^\lambda) - v(x_h) | \leq Ch^{1+\eta}
\end{equation}
for some $\eta>0$. Indeed, recall from \eqref{eqn: Holder} that for any $\beta < \min\{ n-\alpha, 1\}$, we have $\|v\|_{C^{0, \beta}} \leq C$.  So, as $v(x) = \bar{c}$ for some constant $\bar{c}\in\R$ for all $x \in \pa E$,
\[
|v(x_h^\lambda) - v(x_h) | = |v(x_h^\lambda ) - \bar{c} | \leq C \text{dist}\big(x_h^\lambda ,\, \pa E\big)^\beta.	
\]
In order to estimate the right-hand side, we argue separately when $\alpha \in [n-1, n-1/2)$ and when $\alpha \in [n-1/2, n-1/3)$.

First, suppose that $\alpha \in [n-1, n-1/2)$. In this case, by the hypotheses, $\pa E \in C^{1,\gamma}$ for $1+\g>1/(n-\alpha)$. Consider a Taylor expansion of $\vphi$: since $\vphi(0) = |\nabla\vphi(0)|=0$, we have 
\[
\vphi(x' ) = O\big(|x'|^{1+\gamma}\big).
\]
Let  $y_h=he_1 + \vphi(h e_1) e_n \in \pa E$ by \eqref{eqn: graph}. Hence,
\begin{align*}
\text{dist}\big(x_h^\lambda,\, \pa E\big) \leq \big|x_h^\lambda - y_h\big|	 &= |\vphi(-h e_1) - \vphi (he_1)|\\
& \leq |\vphi(-h e_1)| +|\vphi(he_1)| \leq  Ch^{1+\gamma}.
\end{align*}

{Since $1+\g>1/(n-\alpha)$, we may take $\beta<n-\alpha$ large enough such that $(1+\gamma)\beta>1$.} This yields \eqref{eqn: claim1} in this case. 

Next, suppose that $\alpha \in [n-1/2, n-1/3)$. Then, by assumption $\pa E \in C^{2,\gamma}$ {with $2+\gamma>1/(n-\alpha)$}. Now, we use the fact that in the $e_1$ direction, $\pa E$ separates from its reflection as distance to the power $2+\gamma$.
Indeed, a Taylor expansion of $\vphi$ in this case yields
\[
\vphi(x' ) = D^2  \vphi(0)[x',x'] + O\big(|x'|^{2+\gamma}\big).
\]
In this way, again letting $y_h = h e_1 + \vphi(h e_1) e_n$, we see that 
\[
\text{dist}\big(x_h^\lambda,\, \pa E\big) \leq \big|x_h^\lambda - y_h\big|	= |\vphi(h e_1) - \vphi( - h e_1)| \leq  Ch^{2+\gamma}.
\]
Note that in general $\pa E$ separates only quadratically from its reflection across $H_\lambda$; the key point here is that we have chosen our sequence so the $x'$ components of the sequence and its reflection correspond to the reflection $x'=-x'$.
Again, 
{since $2+\gamma>1/(n-\alpha)$, we may choose $\beta<n-\alpha$ large enough so that} $(2+\gamma)\beta>1$, proving \eqref{eqn: claim1}.


\medskip

Next, we claim that if $G_\lambda$ is nonempty, we have
\begin{equation}\label{eqn: claim2}
|v(x_h^\lambda) - v(x_h) | \geq Ch\,.
\end{equation}
Indeed, for fixed $h$, \eqref{eqn: reflection} implies that
\begin{align}\label{a}
v_\lambda(x_h) - v(x_h)& = \int_{G_\lambda} \big|x_h^\lambda -y\big|^{-\alpha} -\big|x_h-y\big|^{-\alpha}\, dy
\end{align}
and the integrand is positive.
In order to establish \eqref{eqn: claim2}, consider the strip $S_\delta = \{ x\,\big\vert\,  |x\cdot e_1|<\delta\}$. Since $G_\lambda$ is an open subset of $H^+_\lambda$, taking $\delta>0$ sufficiently small, we may find some open ball $B$ that is compactly contained in $G_\lambda\setminus S_\delta$. Choosing $h$ sufficiently small, we see that for any $y\in B$, the function $f(t) = |te_1 + \vphi_h e_n -y|^{-\alpha}$ is smooth for $t\in (-h,h)$. 
Thus, we apply the mean value theorem to the function 
$f(t)$ on this interval in order to rewrite the integrand on the right-hand side of\eqref{a} as 
\begin{equation}\label{eqn: mvt}
|x_h^\lambda -y|^{-\alpha} -|x_h-y|^{-\alpha} =2h(  -\alpha|\tilde{x}_h -y|^{-\alpha-2}(\tilde{x}_h -y)\cdot e_1 )
\end{equation}
for some $\tilde{x}_h = te_1 + \vphi_h e_n$ depending on $y$ with $ t\in (-h,h)$. Notice further that $(y-x_h^\lambda)\cdot e_1 \geq \delta/2$ (and hence $(y-\tilde{x}_h)\cdot e_1 \geq \delta/2$) for all $y \in B$, provided $h$ is taken to be small enough. Therefore, for $y \in B$ we have 

\[ 
v_\lambda(x_h) - v(x_h) \geq  -{2h}\alpha\int_{B} |\tilde{x}_h -y|^{-\alpha-2}(\tilde{x}_h -y)\cdot e_1 \, dy \geq C h.
 \]
 This establishes \eqref{eqn: claim2}. Note that when $\alpha \in (0,n-1)$, we can repeat the argument above to obtain that
 \begin{equation}\label{eqn: claim3}
|v(\bar x + he_1) - v(\bar x - he_1) | \geq Ch\,
\end{equation}
 provided $G_\lambda$ is nonempty.
 
 Now, letting $h$ tend to zero,  we see that \eqref{eqn: claim0} (resp. \eqref{eqn: claim1}) and \eqref{eqn: claim3} (resp. \eqref{eqn: claim2}) are in opposition, and so we deduce that $G_\lambda$ is empty. 

For both cases 1 and 2, we have deduced that $E$ is symmetric about $H_\lambda$. Since $e_1$ was chosen arbitrarily, we find that $E$ is a ball.
\end{proof}

\medskip

We now show that Theorem~\ref{thm: not wulff}  follows directly from Theorem~\ref{thm: Riesz char}.

\begin{proof}[Proof of Theorem~\ref{thm: not wulff}]
We first prove statement (i). 
Let $K_r$ be the dilation of the Wulff shape $K$ with $|K_r|=m$, that is, take $r = (m/|K|)^{1/n}$. Direct computation shows that $H^f_{K_r}(x)=(n-1)/r$ for all $x \in \pa K_r$.
Hence, recalling \eqref{eqn: EL}, if $K_r$ is a critical point of \eqref{eqn: aniso drop}, then 
\[
v_{K_r}(x) = \mu' \qquad \text{ for } x \in \pa K_r\,
\]
where $\mu'=\mu-(n-1)/r$. Since $K_r$ is smooth, applying Theorem~\ref{thm: Riesz char}  concludes the proof of (i). 

\medskip

The proof of (ii) similarly follows from Theorem~\ref{thm: Riesz char}. 
To this end, take $\bar m_1$ to be equal to the constant $m_4$ from Theorem~\ref{thm: regularity}.
 Then, if $E$ is a minimizer of \eqref{eqn: aniso drop} for mass $m \leq \bar m_1$,
  then $E$ is of class $C^{2,\beta}$ for all $\beta <n-\alpha$ and $E$ satisfies the Euler-Lagrange equation \eqref{eqn: EL} with $\mu = \mu_E$. 
Then for any $r>0$, setting $F=rE$, we have
	\[
	H_{rE}^f(y) + v_{rE}(y) = r^{-1}H_E^f(y/r) + r^{n-\alpha} v_E(y/r) \qquad \forall y \in \pa F\,.
	\] 
	
	Setting $x=y/r \in \pa E,$ this implies
	\[
	r^{-1}H_E^f(x) + r^{n-\alpha} v_E(x) = r^{-1}\mu_E + (r^{n-\alpha}-r^{-1}) v_E(x) \qquad \forall x \in \pa E\, .
	\]
	So, if $F$ is a critical point of \eqref{eqn: aniso drop} for $m= |F|$, then the left-hand side is equal to a constant $\mu_F$. Rearranging, this yields 
	\[
	v_E(x) = \frac{\mu_F-r^{-1}\mu_E}{r^{n-\alpha}-r^{-1}} = \text{constant}.
	\]
For $\alpha\in (0, n-1/2)$, we may readily apply Theorem~\ref{thm: Riesz char} to conclude that $E$ is a ball. For $\alpha\geq n-1/2$, in order to apply Theorem~\ref{thm: Riesz char} to $E$, we verify that we may find $\beta$ satisfying $2+\beta > 1/(n-\alpha)$  and $\beta<n-\alpha$. Our assumption that $\alpha<n-(\sqrt{2}-1)$ ensures that this is possible.  

Finally arguing as above, we conclude that $H_E^f(x)$ is constant on the boundary of this ball $E$. However, the Alexandrov-type theorem proven in \cite{AnisoAlex} asserts that the only smooth set with constant $H_E^f$ is the Wulff shape. We conclude that the Wulff shape of $f$ is a ball, and thus $f$ is a multiple of the Euclidean norm.
\end{proof}

\section{Crystalline surface tensions: the Proof of Theorem \ref{thm: crystals}}\label{sec:cryst_aniso}
%

In this section, we prove Theorem \ref{thm: crystals},  providing  an example of a surface tension $f$ for which the Wulff shape \emph{is} the minimizer of \eqref{eqn: aniso drop} for sufficiently small mass. 
When $f$ fails to be $C^1$ or elliptic, one can no longer compute the first variation of $\F$ with respect to arbitrary smooth compactly supported variations. A key point becomes understanding among what variations one can compute the Euler-Lagrange equation. As pointed out in the introduction, this is a subtle point and the focus of considerable research in crystalline mean curvature flow. Here we rely on the following 2D structure result from \cite[Theorem 7]{FigalliMaggiARMA}:
%
\begin{theorem}[Figalli-Maggi]\label{thm:cryst_min}
	Let $n=2$ and let $f$ be a crystalline surface tension, so that the Wulff shape is a convex polygon with outer unit normals $\{ \nu_i\}_{i=1}^N.$ Then there exists a positive constant $q_0$ such that if $E$ is a $q$-volume-constrained quasi-minimizer with $q <q_0$, then $E$ is a convex polygon with $\nu_E \in \{ \nu_i\}_{i=1}^N$ for $\Hi^{1}$-a.e. $x \in \pa E$.   
\end{theorem}
This result is in some sense related to that of \cite{taylor_1993} where Taylor considers motion of polygonal curves by crystalline curvature and proves that line segments flow in their normal directions, keeping their same normals. Therefore the segments expand and contract to maintain their directions and adjacencies as they flow towards a steady state.

Now we can prove Theorem \ref{thm: crystals}.

\medskip


\begin{proof}[Proof of Theorem \ref{thm: crystals}]
Let $E$ be a minimizer of the rescaled problem \eqref{eqn: aniso drop 2} with the surface tension $f$ as in \eqref{eqn: f}. It is equivalent to show that there exists $\e'$ such that the Wulff shape is the global minimizer of \eqref{eqn: aniso drop 2} for $\e \leq \e'$.
Recall from Lemma \ref{lem:qmin} that a minimizer $E$ of \eqref{eqn: aniso drop 2} is a $q$-volume-constrained quasi-minimizer of $\F$ with $q=\e c_{n,\alpha}$, so 
the hypotheses of the Theorem~\ref{thm:cryst_min} hold for $\e'$ sufficiently small. Furthermore, since $E$ satisfies the density estimates \eqref{eqn: density ests}, a classical argument shows that $d_H(\pa E, \pa K)\to 0$ as $\e \to 0$, where $d_H$ is the Hausdorff distance. Hence, we deduce further that, up to decreasing $\e'$, we have $\{\nu_E\} = \{ \pm e_1, \pm e_2\}$ for $\Hi^{1}$-a.e. $x \in \pa E$. That is, up to a translation, $E$ is a rectangle of the form
\[
R_a = [-a/2, a/2] \times [-1/2a, 1/2a]
\]
for some $a \in (1-\om(\e),1+\om(\e))$, where $\om(\e)$ is a (nonexplicit) modulus of continuity with $\om(\e)\to 0$ as $\e \to 0$. In this way, we need only to consider the one dimensional family of variations (with respect to $a$) :  $E$ is a minimizer of \eqref{eqn: aniso drop 2} if and only if it is a minimizer of the one-dimensional variational problem
\[
\inf\big\{ \E_{\e,f} (R_a) : R_a = [-a/2, a/2] \times [-1/2a, 1/2a] \big\}\,.
\]
Note that $f(\pm e_1) = f(\pm e_2) = 1/2$. Hence, the energy of $R_a$ is 
\beqns
	\begin{aligned}
\varE_{\e,f}(R_a) &= \F(R_a) + \e \V(R_a) \\
						 &=  \left( a + \frac{1}{a}\right) +\e \, \int_{-a/2}^{a/2} \! \int_{-1/(2a)}^{1/(2a)} \! \int_{-a/2}^{a/2}\! \int_{-1/(2a)}^{1/(2a)} \frac{dx_1 dy_1 dx_2 dy_2}{\big((x_1 - x_2)^2 + (y_1 -y_2)^2\big)^{\alpha/2}}.
	\end{aligned}
\eeqns
Note that
\begin{equation}
	\label{eqn: F}
	\frac{d}{da} \F(R_a) = 1 -\frac{1}{a^2}, \qquad \frac{d^2}{da^2} \F(R_a) = \frac{2}{a^3}.
\end{equation}
Using the change of variables $\tld{x}_i = x_i/a$ and $\tld{y}_i = ay_i$ for $i=1,\,2$, we rewrite the nonlocal term as
	\beqns
		\V(R_a) =  \iiiint_{[-1/2,1/2]^4} \big(a^2(\tld{x}_1-\tld{x}_2)^2+a^{-2}(\tld{y}_1 - \tld{y}_2)^2\big)^{-\alpha/2} \,d\tld{x}_1 d\tld{y}_1 d\tld{x}_2 d\tld{y}_2\,.
	\eeqns
Differentiating with respect to $a$ we see that
	\begin{multline}
		\frac{d}{da}\V(R_a) = \iiiint_{[-1/2,1/2]^4} \Bigg[ \frac{-\alpha}{2}\big(a^2(\tld{x}_1-\tld{x}_2)^2+a^{-2}(\tld{y}_1 - \tld{y}_2)^2\big)^{-\alpha/2-1} \\ \left(2a(\tld{x}_1-\tld{x}_2)^2-\frac{2}{a^3}(\tld{y}_1 - \tld{y}_2)^2\right) \Bigg]\,d\tld{x}_1 d\tld{y}_1 d\tld{x}_2 d\tld{y}_2\, ;\nonumber
	\end{multline}
hence, $(d/da)\Big|_{a=1}\V(R_a)=0$. Pairing this with \eqref{eqn: F} we see that $K$ is a critical point of $\E_{\e,f}(R_a)$. Furthermore, direct computation shows that
	\[
		\frac{d^2}{d a^2}\Big|_{a=1} \E_{\e,f}(R_a)=2-C(\alpha) \e  \geq \frac{1}{2}
	\]
for sufficiently small $\e>0$; hence, $K$ is a stable (in the sense of second variations) critical point of the energy functional $\E_{\e,f}$ for sufficiently small $\e>0$. Now using this stability property we will proceed as in \cite[Theorem 5.1]{AcFuMo13} and \cite[Theorem 2.10]{BoCr14} to show that, up to translations, the Wulff shape $K$ is the unique global minimizer.

Suppose there exists a sequence $\{E_k\}$ of global minimizers of \eqref{eqn: aniso drop 2} with $\e_k \to 0$. Note that, up to translations, $E_k = R_{a_k}$ with $a_k\in (1-\om(\e_k),1+\om(\e_k))$. In particular, $a_k\to 1$ and $E_k \to K$ in $L^1$. Now, since $K$ is a strictly stable critical point of $\E_{\e_k,f}$, for $k>k_0$,
\[
\E_{\e_k,f}(R_{a_k}) \geq \E_{\e_k,f}(K) + \frac{1}{4}(a_k-1)^2 + o\big((a_k-1)^2\big) > \E_{\e_k,f}(K) + \frac{1}{8}(a_k-1)^2\,.
\]
We conclude that $a_k=1$ for $k$ sufficiently large, concluding the proof.
\end{proof}	

\medskip

\begin{remark}[Higher dimensions] \label{rem:higher D}
In their recent work Figalli and Zhang extend the rigidity result of Theorem \ref{thm:cryst_min} to higher dimensions (see \cite[Theorem 1.1]{FiZh2019}). Using this as our starting point, we can argue as in the proof above, and reduce the admissible class of variations to a family defined via a finite number of parameter. To be precise, given
	\[
		f(\nu) = \frac{1}{2}\|\nu\|_{\ell^1(\Rn)} = \frac{1}{2}\left( \sum_{i=1}^n |\nu \cdot e_i| \right)
	\]
$E$ is a minimizer of \eqref{eqn: aniso drop 2} if and only if it is a minimizer of the $(n-1)$-dimensional variational problem
\begin{multline}\nonumber
\inf\Bigg\{ \E_{\e,f} (R_{a_1,a_2,\ldots,a_{n-1}}) : \\
	 R_{a_1,a_2,\ldots,a_{n-1}} = \left[-\frac{a_1}{2}, \frac{a_1}{2}\right] \times \left[ -\frac{a_2}{2}, \frac{a_2}{2}\right] \times \cdots \times \left[-\frac{1}{2a_{1}\cdots a_{n-1}}, \frac{1}{2a_{1}\cdots a_{n-1}}\right] \Bigg\}.
\end{multline}
A calculation of the first and second derivatives with respect to the parameters $a_1,\ldots,a_{n-1}$ shows that the Wulff shape $R_{1,\ldots,1}$ is a stable critical point of the energy $\E_{\e,f}$ for $\e$ sufficiently small, and we can conclude its minimality as above.
\end{remark}

\section{Fully anisotropic models: the Proof of Theorem \ref{thm: wulff min}}\label{sec:doubly_aniso}

In this section we prove Theorem \ref{thm: wulff min}. Before we proceed, let us remark that the problems \eqref{eqn: general} and \eqref{eqn: general confined} admit minimizers when $m$ is sufficiently small. Indeed, existence for \eqref{eqn: general confined} follows from the direct method in the calculus of variations, while existence for \eqref{eqn: general} follows by arguing as in Lemmas \ref{lem: nonopt}, \ref{lem: non opt satisfied}, and Theorem \ref{thm:exist_nonexist}(i), thanks to the subcritical scaling of $\mathcal{U}_i$ with respect to $\F$.

 The interesting feature of these three classes of problems is that each $\mathcal{U}_i(E)$ defined in 
\eqref{eqn: functionals} is maximized by the Wulff shape among sets of a fixed volume, so the variational problems \eqref{eqn: general} and \eqref{eqn: general confined} exhibit the same type of competition between terms as in the classical liquid drop model \eqref{eqn: iso drop}. Let us see that $ \mathcal{U}_1(E)$ is maximized by $K$ among sets $E$ with $|E|=|K|$. Indeed, note that the supremum in $ \mathcal{U}_1(K)$ is attained at $y=0$; up to a translation we may assume the same for $E$. Then
\beqn \label{eqn:Wulff_max}
 \mathcal{U}_1(K) -  \mathcal{U}_1(E) = \int_{K\setminus E} f_*(x)^{-\alpha}\, dx -\int_{E\setminus K} f_*(x)^{-\alpha}\,dx >|K\setminus E| - |E\setminus K| =0
\eeqn
since $f_*(x) >1$ in $E\setminus K$ and $f_*(x)<1$ in $K\setminus E$. The computation is analogous for $ \mathcal{U}_2$ and $ \mathcal{U}_3$.

A key tool in the proof of Theorem~\ref{thm: wulff min} is the following strong form of the quantitative anisotropic isoperimetric inequality which was shown in  \cite[Proposition~1.4]{Neum16}:
	\begin{proposition} \label{Fug-type}
Let $f$ be a  smooth elliptic surface tension with corresponding surface energy $\F$ and Wulff shape $K$. 
Let $E$ be a set such that $|E| = |K|$ and $\bary E = \bary K$, where $\bary E  =|E|^{-1} \int_E x\, dx$ denotes the barycenter of $E$. Suppose 
$$\partial E = \{ x + u(x) \nu_K(x)  \, \big\vert\,  x \in \partial K\}$$
where $u: \partial K \to \mathbb{R}$ is in $ C^1(\partial K)$. There exist $C$ and $\mu$ depending on $n$ and $f$  such that if $\| u \|_{C^1(\partial K)}\leq \mu$, then
 \begin{equation}\label{fugnew}
 C\|u\|_{H^1(\partial K)}^2 \leq  \F(E) - \F(K).
\end{equation}
\end{proposition}

\medskip

The idea of the proof of Theorem~\ref{thm: wulff min} the following: Using the Wulff shape as a competitor in \eqref{eqn: general} and applying Proposition~\ref{Fug-type}, we find that  $\mathcal{U}_i(K)-\mathcal{U}_i(E)$ is bounded below by a constant (increasing with $m^{-1}$) multiple of the squared distance between $E$ and $K$. On the other hand, the criticality and $C^2$ bound of $\mathcal{U}_i(K)$ ensure that this gap is at most a fixed constant multiple of the squared distance.
 These ideas are in the spirit of those used in \cite{FFMMM,KM14} in the setting of \eqref{eqn: iso drop}.

\medskip

\begin{proof}[Proof of Theorem~\ref{thm: wulff min}]
Let us start by observing the scaling of the energies:
\[
\begin{array}{lll}
	\mathcal{U}_1(rE)& = r^{\g_1} \mathcal{U}_1(E) & \qquad \g_1 = n-\alpha\, , \\
	\mathcal{U}_2(rE) &= r^{\g_2} \mathcal{U}_2(E) & \qquad \g_2 = n+\beta\, , \\
	\mathcal{U}_3(rE)& = r^{\g_3} \mathcal{U}_3(E) +r^n\log(r)|E| & \qquad \g_3 = n\, .
\end{array}
\]
Thus in place of \eqref{eqn: general} and \eqref{eqn: general confined} respectively, we set $\e_i = (m|K|^{-1})^{(\g_i - n+1)/n}$ and study the equivalent variational problems
\begin{equation}\label{eqn: rescaled}
\inf \Big\{ \F(E) + \e_1 \, \mathcal{U}_1(E)  \, \big\vert\,  |E| =|K|\Big\}
\end{equation}
and 
\begin{equation}\label{eqn: rescaled 2}
\inf\Big\{ \F(E) + \e_i \, \mathcal{U}_i(E)  \, \big\vert\,  |E| =|K|, \ E\subset B_{c_{n,f}} \Big\}\, \qquad i=2,3\,.
\end{equation}
To prove Theorem~\ref{thm: wulff min}, it is equivalent to show that for $i=1,2,3$, there exists $\e'_i=\e'_i(n, f, \mathcal{U}_i)$ such that for $\e\leq \e'_i$, the only  minimizers of \eqref{eqn: rescaled} and \eqref{eqn: rescaled 2} are Wulff shapes.

\bigskip

\noindent {\it Claim:} For $i=1,2,3$, suppose $E_i$ is a global minimizer of \eqref{eqn: rescaled} or \eqref{eqn: rescaled 2}. Then, up to a translation and provided $\e_i$ is sufficiently small, $\pa E_i$  is  a small $C^{1}$ graph over $\pa K$: 
\[
 \pa E_i = \{ x + u_i(x) \nu_K(x)  \, \big\vert\,  x \in \pa K \},
 \]
where $u_i: \pa K \to \mathbb{R}$ has $\|u_i\|_{{C^1}(\pa K)}\leq \mu/3$ with $\mu$ as in Proposition~\ref{Fug-type}. 

\bigskip

Let us assume the claim for now, since the proof differs between \eqref{eqn: rescaled} and \eqref{eqn: rescaled 2}. Up to replacing $\mu/3$ by $\mu$, we may further assume that $\bary E_i = \bary K.$  We can therefore apply Proposition~\ref{Fug-type} to find that
\begin{equation}\label{eqn: quantitative}
C\| u_i\|_{H^1(\pa K)}^2 \leq \F(E_i) - \F(K)
\end{equation}
with the constant $C>0$ given in \eqref{fugnew}.
Now, taking the Wulff shape as a competitor in \eqref{eqn: rescaled} and \eqref{eqn: rescaled 2} and rearranging the energy, we find that
\[
\F(E_i) -\F(K) \leq \e_i \big( \mathcal{U}_i(K) - \mathcal{U}_i(E_i)\big).
\]
Letting $X_i = u_i \nu_K$, a Taylor expansion of $\mathcal{U}_i(E_i)$ yields 
\begin{align}\label{eqn: exp}
\mathcal{U}_i(E_i) &= \mathcal{U}_i(K) +\delta\, \mathcal{U}_i(K)[X] +\frac{1}{2} \delta^2\,\mathcal{U}_i(K)[X, X]+\mu\, O(\|u\|_{H^1(\pa K)}^2)\,.
\end{align}
Here $\delta \, \mathcal{U}_i(K)$ and $\delta^2\, \mathcal{U}_i(K)$ denote the first and second variations respectively.
Direct computation (recalling that $f_*(x) = 1$ for $x \in \pa K$) shows that the first variations are given by
\begin{equation} \label{eqn: first var}
\delta\, \mathcal{U}_1(K)[X_1] = \int_{\pa K} u_1 \, d\Hi^{n-1} \,, \quad
\delta\, \mathcal{U}_2(K)[X_2] = \int_{\pa K} u_2 \, d \Hi^{n-1}\,,\quad
\delta\, \mathcal{U}_3(K)[X_3] = 0\,.
\end{equation}
Arguing as in \cite[proof of Lemma 4.1]{Neum16}, the fact that $|E_i| = |K|$ implies that  
\begin{equation}\label{eqn: first var 2}
\int_{\pa K} u_i \, d\Hi^{n-1} = - \frac{1}{2} \int_{\pa K } H_K u_i^2 \, d\Hi^{n-1} + \mu \, O(\|u_i\|_{H^1(\pa K)}^2)
\end{equation}
where $H_K$ is the (isotropic) mean curvature of $K$.
Next, a somewhat lengthy yet standard computation (making use of the identity $\na f_*(x) = \nu_K(x) /f(\nu_K(x))$ for $x \in \pa K$) shows that 
\begin{equation} \begin{aligned} \label{eqn: second var}
	\delta^2\mathcal{U}_1(K)[X_1,X_1] & \geq \int_{\pa K} u_1^2 H_K + u_1 \na u_1 \cdot {\nu_K} - \alpha u_1^2 f(\nu_K)^{-1}\, d\H^{n-1} \\
														& \geq -\kappa_1 \|u_1\|_{H^1(\pa K)}^2\, ,\\
	\delta^2\mathcal{U}_2(K)[X_2,X_2] & \geq \int_{\pa K} u_2^2 H_K + u_2 \na u_2 \cdot {\nu_K} +\beta u_2^2 f(\nu_K)^{-1} \, d\H^{n-1} \\
														&\geq -\kappa_2 \|u_2\|_{H^1(\pa K)}^2\, ,\\
	\delta^2\mathcal{U}_3(K)[X_3,X_3] & \geq \int_{\pa K} u_3^2 f(\nu_K)^{-1}\, d\H^{n-1} \geq -\kappa_3 \|u_3\|_{H^1(\pa K)}^2
\end{aligned}
\end{equation}
for some constants $\kappa_i >0$.

Together, \eqref{eqn: exp}, \eqref{eqn: first var}, \eqref{eqn: first var 2}, and \eqref{eqn: second var} imply that 
\begin{equation}\label{eqn: upper bound}
\mathcal{U}_i(K) - \mathcal{U}_i(E_i) \leq C_i \| u_i\|_{H^1(\pa K)}^2.
\end{equation}
where $C_i = C(n,f,\mathcal{U}_i)$. Finally, combining \eqref{eqn: quantitative} and \eqref{eqn: upper bound}, we see that taking $\e_i$ sufficiently small forces $u_i= 0$, and hence $E_i$ is the Wulff shape.

\bigskip

Let us now prove the claim, arguing separately for $\mathcal{U}_1(E)$ and for $\mathcal{U}_2(E)$, $\mathcal{U}_3(E)$.
Let $y_E \in \R^n$ be a point attaining the supremum in $\mathcal{U}_1(E)$. Then, using \eqref{eqn:Wulff_max}, for any $E, F$ with $|E| = |F|$ we have
	\[
	\mathcal{U}_1(E) - \mathcal{U}_1(F)\leq \int_{E\triangle F} f_*(x-y_E)^{-\alpha} \,dx\leq \int_{rK} f_*(x)^{-\alpha}\, dx  
	\]
	where $r$ is chosen so that $|rK| = |E\triangle F|$. The term on the right-hand side is equal to $\frac{n}{n-\alpha}|K| r^{n-\alpha}$.
	In this way, arguing as in the proof of Lemma~\ref{lem:qmin}, we find that a minimizer $E_1$ of \eqref{eqn: general} satisfies 
	\[
	\F(E_1) \leq \F(F) + \Lambda\, |E_1\triangle F|^{\g_1/n}\, \qquad \text{ whenever } E_1\triangle F\cc B_r(x)
	\]
	for some $x \in R^n$ and $r \leq \ell_f/L_f$. =
	 Since $\g_1 > n-1$, the results of \cite{Tamanini82} show that $\pa^*E$ is locally $C^{1,\eta}$ for $\eta= \frac{1-\alpha}{2}$. Then, using the first variation of the energy given by
	\[
		H_{E_1}^f(x) + \e_1\,f_*(x)^{-\alpha}=\text{const} \qquad \forall x\in\pt^* E_1,
	\]
	we can repeat Steps 2 and 3 of the proof of Theorem~\ref{thm: regularity} to conclude the claim in this case.

Now, for any $E, F \subset B_{c_{n,f}}$, we have
\[
\mathcal{U}_2(E) - \mathcal{U}_2(F) \leq \int_{E\triangle F} f_*(x-y_E)^\beta \, dx \leq (2c_{n,f}/\ell_f)^\beta  |E\triangle F|.
\]
Similarly,
\[
\mathcal{U}_3(E) -\mathcal{U}_3(F) \leq \int_{E\triangle F} \log (f_*(x-y_F)) \, dx \leq \log(2c_{n,f}/\ell_f) |E\triangle F|.
\]
As such, we may argue as in Lemma~\ref{lem:qmin} to find that if $E_i$ is a minimizer of \eqref{eqn: rescaled 2} for $i=2,3$, then $E_i$ is a quasi-minimizer of the surface energy. Again, using the Euler-Lagrange equations
	\begin{gather*}
		H_{E_2}^f(x) - \e_2\,f_*(x)^{\beta}=\text{const} \qquad \forall x \in \pt^* E_2,\\
		H_{E_3}^f(x) - \e_3\,\log f_*(x)=\text{const} \qquad \forall x \in \pt^* E_3,
	\end{gather*}	
we may repeat Steps 2 and 3 of the proof of Theorem~\ref{thm: regularity} to conclude the claim.
\end{proof}

\section{Open Problems}\label{open} 
We conclude by mentioning and recalling some important open problems. First note that all the open problems for the liquid drop problem 
(\ref{eqn: iso drop}) carry over to (\ref{eqn: aniso drop}). In the case of smooth anisotropies, it is not clear what more one could hope to prove about minimizers of (\ref{eqn: aniso drop}), with a clear characterization unlikely. Here, numerical computations could prove quite insightful in a qualitative assessment of the difference between the shape of minimizers  and the  Wulff shape. For crystalline anisotropies there remains much to be done in determining the minimality of the Wulff shape. The most tractable problem would be to generalize Theorem \ref{thm: crystals} to all regular polygons in 2D. 

We would like to end by highlighting the open problem alluded to in Remark \ref{remark-1}. This problem is, in our opinion, a rather fundamental problem \eqref{anisotropic capacity problem} which pertains to 
 the   ``equilibrium figure" with an anisotropic potential. Unfortunately,  given that Riesz rearrangement techniques fail, it is not clear what techniques one could employ.

\subsection*{Acknowledgements} The authors would like to thank the anonymous reviewers for their careful reading and their insightful comments. RC was supported by an NSERC (Canada) Discovery Grant. RN was supported in part by the National Science Foundation under Grant No. DMS-1502632 ``RTG: Analysis on manifolds'' at Northwestern University and Grant No. DMS-1638352 at the Institute for Advanced Study.

\bigskip

\appendix

\section{Proof of Lemma~\ref{lem:qmin}}\label{app: proof of q min}

Let us prove Lemma~\ref{lem:qmin}.
\begin{proof}
Taking $F$ to be any competitor with $|E|=|F|=1$, \eqref{eqn: lip} implies that
\[
\F(E) \leq \F(F) + \e \big(\V(E) - \V(F)\big) \leq \F(F) + c_{n,\alpha}\e\, |E\triangle F|.
\]
This proves (i). To show (ii), we employ the typical trick of showing that $E$ is a minimizer of the unconstrained variational problem
\begin{equation}
	\label{eqn: unconstrained problem}
\inf \left\{\E_{\e,f}(F) + Q \big| |F|-1 \big| \, \big\vert\,  F \subset \R^n\right\}.
\end{equation}
for some fixed constant $Q =5n\bar \E_{\e,f},$ where we adopt the shorthand $\bar \E_{\e,f}= \E_{\e,f}(E)$. It suffices to show that if
	\begin{equation}\label{eqn: small2}
		\E_{\e,f}(F) + Q \big| |F|-1 \big| \leq \bar\E_{\e,f},
	\end{equation}
then $|F|=1$. For any $F$ satisfying \eqref{eqn: small2}, let $G=r F$ so that $|G|=1$. We note immediately that $r\in[1,2)$. Indeed, 
\begin{align*}
	\E_{\e,f}(G)<\E_{\e,f}(F)\leq \bar \E_{\e,f} & \qquad \qquad \text{ if }r<1,\\
	 Q/2 \leq Q(1-r^{-n}) =Q ||F|-1| \leq \bar \E_{\e,f}  & \qquad \qquad \text{ if } r\geq 2,
\end{align*}
violating the minimality of $E$ in \eqref{eqn: aniso drop 2} and our choice of $Q$ respectively. Now, since $n-1<2n-\alpha$, we have
\[
\E_{\e,f}(F) \geq r^{\alpha -2n}\E_{\e,f}(G) \geq r^{\alpha -2n}\bar\E_{\e,f}
\]
and so rearranging \eqref{eqn: small2} gives
\[
Q(1-r^{-n}) \leq \bar\E_{\e,f}(1-r^{\alpha-2n})\,.
\]
By concavity, we have the bounds $1-r^{-n} \geq (r-1)/2$ and $1-r^{\alpha -2n} \leq 2n(r-1)$ for $r\in [1,2)$.
Thus,
\[
Q(r-1) \leq 4n\bar\E_{\e,f}(r-1),
\]
forcing $r=1$ by our choice of $Q.$ We conclude that $E$ is a minimizer of \eqref{eqn: unconstrained problem}.

Hence, taking any $F$ with $|E\triangle F|\leq 1$ (and hence $|F|\leq 2$) as a competitor in \eqref{eqn: unconstrained problem} and recalling \eqref{eqn: lip}, we obtain
	\beqns
		\begin{aligned}
			\F(E) &\leq \F(F) + \e\, \big( \V(F) - \V(E) \big) + Q \big| |F|-1 \big| \\
					&\leq \F(F) + (\e\,c_{n,\alpha}+Q) |E \triangle F|
		\end{aligned}
	\eeqns
 Thus, $E$ is a $(\Lambda,1)$-quasi-minimizer with $\Lambda=\e\,c_{n,\alpha}+Q$. If $\e\leq 1$, then $\Lambda$ can be taken to be independent of $\e$. 
\end{proof}

\section{Regularity of minimizers}\label{appendix: regularity}

Next we outline the proof of the regularity result.

\begin{proof}[Proof of Theorem~\ref{thm: regularity}]
As before, it will be convenient to consider \eqref{eqn: aniso drop 2} in place of the equivalent problem \eqref{eqn: aniso drop}; by rescaling, the same statements will hold for minimizers of \eqref{eqn: aniso drop}. We also assume without loss of generality that we have multiplied $f$ by a constant so that the Wulff shape $K$ has unit mass.

\bigskip

\noindent {\it Step 1: Quasi-minimality and $C^{1,\gamma}$ regularity.}
Let $E$ be a minimizer of \eqref{eqn: aniso drop 2}. Then, by Lemma \ref{lem:qmin}, $E$ is a $q$-volume-constrained quasi-minimizer of $\F$.
 The epsilon-regularity theory for quasi-minimizers of $\F$,  see \cite{alm66, SSA77, bomb82, DuzaarSteffen02}, ensures that $\pa^*E$ is of class $C^{1,\g}$ for $\g \in (0,1)$. To state this more precisely, let us introduce a bit of notation.
For $x \in \Rn$, $r>0$, and $\nu \in \S^{n-1}$, we define
\begin{align*} 
\mathbf{C}_{\nu}(x,r) &= \{ y\in \mathbb{R}^{n}  \, \big\vert\,  |p_{\nu} (y-x)|<r, \ |q_{\nu}(y-x) < r\},\\
\mathbf{D}_{\nu} (x,r) &= \{ y \in \Rn  \, \big\vert\,  |p_{\nu}(y-x)| <r , \  |q_{\nu}(y-x) | = 0\},
\end{align*}
where $q_{\nu}(y) =  y \cdot \nu$ and $p_{\nu}(y) = y - (y \cdot \nu) y.$
We then define the \textit{cylindrical excess} of $E$ at $x$ in direction $\nu$ at scale $r$ to be
$$\text{exc} (E, x ,r,\nu ) = \frac{1}{r^{n-1}} \int_{\mathbf{C}_{\nu}(x,r) \cap \partial^* E}  \frac{|\nu_E - \nu|^2}{2} \,d\Hn\,.$$
 For all $\g\in(0,1)$ there exist constants $C(n,f,\g)$ and $\delta$ depending on $n, f$ and $\g$ such that if a quasi-minimizer $E$ satisfies
\begin{equation}\label{eqn: small excess}
{\rm{{exc}}}(E, x,r,\nu) + q\, r <\delta
\end{equation}
then there exists $u\in C^{1,\g}(\mathbf{D}_{\nu}(x,r))$ with $u(x) =0$ such that
\begin{equation}\label{eqn: graph scale}
\begin{split}
 \mathbf{C}_{\nu}(x,r/2) \cap \partial^*E & = (\id + u \nu)(\mathbf{D}_{\nu}(x,r/2)),\\
  \|u\|_{C^0(\mathbf{D}_\nu(x, r/2))}& < C(n,f,\g)\,r\, {\rm{{exc}}}(E,x,r,\nu)^{1/(2n-2)},\\
  \|\nabla u\|_{C^0(\mathbf{D}_{\nu}(x, r/2))}&< C(n,f,\g)\, {\rm{{exc}}}(E,x,r,\nu)^{1/(2n-2)},\\
  r^{\g} [\nabla u ]_{C^{0,\g}(\mathbf{D}_{\nu}(x,r/2))}& < C(n,f,\g)\, {\rm{{exc}}}(E,x,r,\nu)^{1/2}.
\end{split}
\end{equation} 
For any $x \in \pa^*E$, \eqref{eqn: small excess} will be satisfied at sufficiently small scale $r$, and we conclude that $\pa^*E$ is locally a $C^{1,\g}$ hypersurface.

\bigskip

\noindent {\it Step 2: $C^{2,\beta}$ regularity.}  Furthermore, as we noted in equation \eqref{eqn: EL}, $E$ satisfies the Euler-Lagrange equation
\[
H_E^f(x) + \e\, v_E(x) = \mu\  \qquad \forall x\in \pa^*E
\]
in a distributional sense. Given $x\in \pa^*E$, if we choose $r$ to be suitably small so that \eqref{eqn: small excess} holds, the Euler-Lagrange equation reads
\begin{equation}\label{eqn: EL app}
\DIVV'(\na' f'(\na'u(z)) = \e v_E(z,u(z)) -\mu \qquad \forall z \in \mathbf{D}_{\nu}(x,r/2).
\end{equation}

Here, $f'(z) = f(-z,1)$ and $\DIVV'$ and $\na'$ indicate derivatives with respect to the $z$ variable. Applying Schauder estimates (see \cite[Theorem~6.2]{GT}), 
\begin{equation}\label{eqn: c2beta}
\| u\|_{C^{2,\beta}(\mathbf{D}_\nu(x,r/4))} \leq C(r,f, \| u\|_{C^{1,\g}(\mathbf{D}_\nu(x,r/2)},\beta) \| v_E-\mu\|_{C^{0,\beta}(\mathbf{D}_\nu(x,r/2))}\,.
\end{equation}
Finally, \eqref{eqn: Holder} shows that $\| v_E\|_{C^{0,\beta}} \leq C(n,\alpha, \beta ,|E|)$  for all $\beta\in(0,  \min\{1, n-\alpha\})$. 
 This concludes (i).

\bigskip

\noindent {\it Step 3: Improved convergence.} To establish (ii) and (iii), consider a sequence of minimizers $E_\e$ of \eqref{eqn: aniso drop 2} with $\e \to 0$. Provided $\e \leq 1$, $r_0$ and $q$ in \eqref{eqn: qmin} can be taken to be independent of $\e$, and so $E_\e$ satisfy uniform (in $\e$) density estimates. Thanks to the diameter bound \eqref{eqn: diameter bound} and the Wulff inequality \eqref{eqn: wulff inequality}, $E_\e \to K$ in $L^1$, and thanks to the density estimates, $d_H(\pa E_\e , \pa K)\to 0,$ where $d_H$ is the Hausdorff distance. We argue as in \cite{CicaleseLeonardi} to obtain a uniform graphicality scale $r$ on which: 
\begin{gather}
\nonumber	  \partial K \cap \mathbf{C}_{\nu}(x,r/2) = (\id + u \nu)(\mathbf{D}_{\nu}(x,r/2)), \\
\nonumber	  \partial E_\e \cap \mathbf{C}_{\nu}(x,r/2)  = (\id + u_{\e} \nu)({\mathbf{D}}_{\nu}(x,r/2)),\\
\label{eqn: convergence}	  \| u_\e -u\|_{L^\infty}  \to 0.
\end{gather}
and the estimates of \eqref{eqn: graph scale} hold uniformly in $\e$.
The key point here is the continuity of the cylindrical excess with respect to $L^1$ convergence of quasi-minimizers with uniform $r_0$ and $q$; since $\pa K$ is smooth, $E_\e$ will satisfy the flatness assumption \eqref{eqn: small excess} for every $x \in \pa E_\e$ and at a uniform scale in $\e$.

Finally, since
\[
\mu_\e = \frac{1}{n}\Big((n-1) \F(E_\e) + \e (2n-\alpha)\V(E_\e)\Big) \leq C(\F(K) + \V(K)),
\]
we see that the right-hand side of \eqref{eqn: c2beta} is bounded by a constant independent of $\e$.
By the Arzel\`a-Ascoli theorem, $\| u_\e -u\|_{C^{2,\beta'}} \to 0$ for all $\beta' \leq \beta.$ In particular, for $\e$ sufficiently small,  this yields (ii).  Since $K$ is uniformly convex, it follows that $\pa E_\e$ is uniformly convex as well provided $\e$ is sufficiently small.
\end{proof}

\begin{remark}[Higher regularity]
	{\rm Of course, starting from the $C^{1,\gamma}$ regularity coming from quasi-minimality, the regularity of $\pa^*E$ (and then $\pa E$ for small masses) can be improved much as the Euler-Lagrange equation will allow. Indeed, in place of Theorem~\ref{thm: regularity}(i), one can prove that the reduced boundary $\pa^*E$ is a $C^{2+k,\beta}$ hypersurface for all $\beta\in (0,1)$ and $k+\beta<n-\alpha$.
	To establish higher regularity, one differentiates \eqref{eqn: EL app} and applies the same Schauder estimates \eqref{eqn: c2beta} to derivatives of $u$, making use of the smoothness of $f$ and \eqref{eqn: Holder}.

	}
\end{remark}

\begin{remark}[Lower regularity assumptions on $f$]
	{\rm  For convenience we have assumed that $f \in C^{\infty}(\R^n\setminus \{0\})$ throughout the paper. Provided $f \in C^{2, \gamma}(\R^n\setminus \{0\})$, one can still establish Theorem~\ref{thm: regularity}(i) for $\beta<\min\{\gamma, n-\alpha\}$. The proof follows the one given above verbatim. In fact, to establish the $C^{1,\gamma}$ regularity of $\pa^*E$, we only need $f$ to be elliptic with $f \in C^{1,1}(\R^n\setminus\{0\})$, this follows from the results of \cite{FigalliJDG}.
	}
\end{remark}

\begin{remark}[Quantitative estimates]
	{\rm
	In the case $\alpha\in (0,n-1)$, one can adapt ideas from \cite{FigalliMaggiARMA} to make parts (ii) and (iii) of Theorem~\ref{thm: regularity}  quantitative in terms of the mass: there is a critical mass $m_6=m_6(n,f,\alpha, \beta_0)$ and a constant $C=C(n,f,\alpha, \beta_0)$ such that under the hypotheses of Theorem~\ref{thm: regularity} with $m\leq m_6,$ and setting $
F  = ( |K|/m)^{1/n} E,$
\[
\max_{\pa F} |\nabla^2 f(\nu_F) \na \nu_F - \id_{T_x\pa F}| \leq C m^{2\beta/(n+2\beta)}.
\]
Such quantitative estimates were shown in \cite[Theorem 2]{FigalliMaggiARMA} for a related class of problems; while our nonlocal repulsion term $\V(E)$ does not fall into the class of potential terms studied there, their proof extends to our setting for $\alpha\in(0,n-1)$ with only minor adjustments. The only notable difference comes in the study of the second variation, where one must bound an additional term in the second variation of $\V(E)$ that does not appear for the potentials studied in \cite{FigalliMaggiARMA}. In the case $\alpha \in [n-1,n)$, there are some subtle integrability issues for the second variation of $\V(E)$ and we do not know if the estimates can be made quantitative in this case.
}	
\end{remark}

\section{Auxiliary lemmas toward existence and diameter bound}\label{app: existence}

The following ``non-optimality criterion'' of \cite{KM14}, which follows by direct comparison and the Wulff inequality, is key in establishing both existence and the diameter bound.

\begin{lemma}\label{lem: nonopt}
	Fix $n\geq 2$, $\e>0$ and a surface tension $f$ with $\ell_f$ and $L_f$ given by \eqref{eqn:l and L}. There exists $\delta_0=\delta_0(n,\alpha, \ell_f, L_f)$ such that the following holds. Suppose $|F| =1$ and $\E_{\e,f}(F) \leq 2 \bar \E_{\e,f}$. If $F=F_1\cup F_2$ for two nonempty disjoint sets $F_1, F_2$ with 
	\begin{equation}\label{eqn: small gain}
\F(F_1) + \F(F_2) -\F(F) \leq  \F(F_2)/2
\end{equation}
and
\begin{equation}\label{eqn: small}
|F_2| \leq  \delta_0\, \min \{ 1, \e^{-n/(n+1-\alpha)}\},
\end{equation}
then $\E_{\e,f}(\hat{F}_1)<\E_{\e,f}(F)$, where $\hat{F}_1= rF_1$ such that $|\hat{F}_1|=1$.
\end{lemma}
\begin{proof}
Let $\delta=|F_2| \in (0,1)$, so that $\hat{F}=(1-\delta)^{1/n}F_1.$ Observe that \eqref{eqn: small gain} implies that $\E_{\e,f}(F_1) \leq \E_{\e,f}(F) - \F(F_2)/2$. Hence, by \eqref{eqn: small gain},
\begin{align*}
	\E_{\e,f}(\hat{F}) & \leq (1+\delta)^{(2n-\alpha)/n}\E_{\e,f}(F_1)\\
					& \leq (1+C(n,\alpha)\,\delta)\E_{\e,f}(F_1)\\
					& \leq (1+C(n,\alpha)\,\delta)\Big(\E_{\e, f}(F) - \F(F_2)/2\Big).
\end{align*}
So, thanks to the Wulff inequality \eqref{eqn: wulff inequality}, we find that $\E_{\e,f}(\hat{F})-\E_{\e, f}(F)$ is bounded above by 
\[
C(n,\alpha)\,\delta \bar \E_{\e, f} -\delta^{(n-1)/n}n|K|^{1/n}\, ,
\]
which, recalling \eqref{eqn: energy bound 1} and \eqref{eqn: energy bound 2}, is strictly negative provided $\delta<\delta_0$ for some $\delta_0>0$ sufficiently small.
\end{proof}

\begin{lemma}\label{lem: non opt satisfied}
Fix $n\geq 2$ and a surface tension $f$ with $\ell_f$ and $L_f$ given by \eqref{eqn:l and L}. There exists $\e_1(n,\alpha, \ell_f, L_f) \leq 1$ and $\rho_0$ such that the following holds. Let $F$ be a set of finite perimeter with $|F|=1$ and $\E_{\e, f}(F)\leq \bar\E_{\e,f}+ \e$. Then, for some $\rho \in [L_f/(\ell_f\om_n^{1/n}),\rho_0)$ and after a translation, the sets 
\[
F_1 = F\cap B_\rho \qquad \text{ and } \qquad F_2 = F \setminus B_\rho
\]
satisfy \eqref{eqn: small gain} and \eqref{eqn: small}.
\end{lemma}
\begin{proof}
Let $F$ be a set of finite perimeter with  $|F|=1$ and $\E_{\e,f}(F)\leq \bar \E_{\e,f}+ \e$. Set $s = |K|^{-1/n}$, so that $|K_s|=1$, and replace $F$ with a translation $F+x_0$ such that $|(F+x_0)\triangle K_s| =\inf_x|(F+x)\triangle K_s|$. Note that $K_s \subset  B_\rho $ provided $\rho\geq L_f/(\ell_f\om_n^{1/n})$. For all such $\rho$, set 
\[
F_1^\rho = F\cap B_\rho \qquad \text{ and } \qquad F_2^\rho = F \setminus B_\rho.
\]
We claim that there exists a constant $\rho_0=\rho_0(n,\ell_f, L_f)$ such that \eqref{eqn: small gain} and \eqref{eqn: small} are satisfied for some $\bar \rho\leq \rho_0$ provided $\e_1$ is sufficiently small. Let us first see that \eqref{eqn: small} is satisfied for every $\rho\geq L_f/(\ell_f\om_n^{1/n})$, provided $\e_1$ is small enough. Indeed,
note that 
\[
\F(F) - \F(K_s) \leq \e (\V(K_s) - \V(F)) + \e \leq C(n,f,\alpha)\, \e\,.
\]
Then, since $F_2^\rho \subset E\triangle K_s$ for every , we have  $|F_2^\rho| \leq |E\triangle K_s| \to 0$ as $\F(F) - \F(K_s) \to 0$. So, for any such $\rho$, we have that $|F_2^\rho|$ satisfies \eqref{eqn: small} for $\e_1$ sufficiently small.

%
Now, let $\bar \rho$ be the smallest constant greater than or equal to $L_f/(\ell_f\om_n^{1/n})$ such that \eqref{eqn: small gain} is satisfied. 
For a.e. $\rho>0$, we have
 \begin{equation}\label{eqn: splitting}
 \F(F_1^\rho) + \F(F_2^\rho) -\F(F) = \int_{\pa B_\rho \cap F} f(\nu_{B_\rho}) + f(-\nu_{B_\rho})\,d\H^{n-1} \leq 2L_f \H^{n-1}(\pa B_\rho\cap F)\,;
 \end{equation}
see \cite[Theorem 16.3 and Proposition 2.16]{Maggi2012}.
Thus, by the definition of $\bar \rho$, for a.e. $\rho \in [L_f, \bar \rho)$ we have 
 \[
 \H^{n-1}(\pa B_\rho\cap F) \geq \F(F_2^\rho)/ 4L_f.
 \]
 Define the function $U(\rho ) = |F\setminus B_\rho| = |F_2^\rho|$. Then for a.e. $\rho\in [L_f, \bar \rho)$,
 \[
 U'(\rho) = -\H^{n-1}(\pa B_\rho\cap F) \leq - \F(F_2^\rho)/ 4L_f \leq -c_1  |F_2^\rho|^{(n-1)/n} =- c_1 U(\rho)^{(n-1)/n}\, ,
 \]
 where $c_1:=n|K|^{1/n}/4L_f$.
Integrating from $L_f$ to $\bar \rho$ and noting that $U(L_f)\leq C(n,L_f)$, we find that $\bar \rho \leq \rho_0(n,\ell_f, L_f)$.
This concludes the proof.
\end{proof}

\bigskip

Let us now prove the diameter bound given in \eqref{eqn: diameter bound}.

\begin{proof}[Proof of \eqref{eqn: diameter bound}]
It is equivalent to show that a minimizer of \eqref{eqn: aniso drop 2} satisfies
\[
\diam E \leq C(n, \alpha, \ell_f, L_f) \e^{(n-1)/(n+1-\alpha)}.
\]
Thanks to Section~\ref{subsec: q min}, it suffices to show that for any $x\in E$, the lower density estimate of \eqref{eqn: density ests} holds for all $x\in E$  up to scale $\bar r\geq c_1\,\e^{-1/(n+1-\alpha)}$ with the constant $c_1$ as in the previous proof. 

Fix $x\in E$ and let $F_1^r= E\setminus B_r(x)$ and $F_2^r = E \cap B_r(x)$.
Let $\bar r$ be the smallest $r>0$ such that 
\begin{equation}\label{eqn: small gain 2}
\F(F_1^r) + \F(F_2^r) - \F(E) \leq \F(F_2^r)/2	\, .
\end{equation}
Note that $\bar r>\om_n^{-1/n}\delta_0 \e^{-1/(n+1-\alpha)}$, otherwise we may apply Lemma~\ref{lem: nonopt} to contradict the minimality of $E$. Note also that $F_1^{\bar r}$ is nonempty. Indeed, if not, then $E\subset B_{\bar r}(x)$ (up to a null set), and in particular $\bar r \geq \om_n^{-1/n}$. In this case, we have
\[
\E_{\e,f}(E) \geq \e \V(E) \geq \e 2^{-\alpha}\om_n^{\alpha/n},
\]
contradicting \eqref{eqn: energy bound 2}.

Now, \eqref{eqn: small gain 2} allows us to extend the usual proof of lower density estimates of \eqref{eqn: density ests} up to scale $\bar r.$
Indeed, for a.e. $r>0$, we have
 \[
 \F(F_1^r) + \F(F_2^r) -\F(E) = \int_{\pa B_r \cap E} f(\nu_{B_r}) + f(-\nu_{B_r})\,d\H^{n-1} \leq 2L_f \H^{n-1}(\pa B_r\cap E)\,,
 \]
and for any $r<\bar r$, the left-hand side is bounded above by $\F(F_2^r)/2$, so
 \[
 \H^{n-1}(\pa B_r\cap E) \geq \F(F_2^r)/ 4L_f
 \]
 for a.e. $r \in (0, \bar r)$. Hence, setting $U(r) = |E\cap B_r| = |F_2^r|$, we have for a.e. $r\in (0, \bar r)$
 \[
 U'(r) = \H^{n-1}(\pa B_r\cap E) \geq \F(F_2^r)/ 4L_f \geq c_1  |F_2^r|^{(n-1)/n} = c_1 U(r)^{(n-1)/n}\, ,
 \]
 where $c_1=n|K|^{1/n}/4L_f\geq n\om_n^{1/n} \ell_f/4L_f$.
Integrating from $0$ to $r$ for any $r\leq \bar r$ completes the proof.
\end{proof}

\bibliographystyle{IEEEtranS}
\def\url#1{}
\bibliography{references2}

\end{document}